\documentclass[12pt,twoside]{article}
\usepackage{amsmath,amsthm}
\usepackage{amssymb,latexsym}
\usepackage{mathrsfs}
\usepackage{amscd}
\usepackage{enumerate}
\usepackage{color}
\usepackage{cite}
\usepackage{indentfirst}
\usepackage{amssymb}
\usepackage[mathscr]{euscript}
\usepackage{graphicx}
\usepackage{titletoc}
\usepackage{hyperref}

\titlecontents{section}[0pt]{\addvspace{3pt}\filright}
              {\contentspush{\thecontentslabel\quad}}
              {}{\titlerule*[8pt]{$\cdot$}\contentspage}

\newtheorem{theorem}{Theorem}[section]
\newtheorem{corollary}[theorem]{Corollary}
\newtheorem{lemma}[theorem]{Lemma}
\newtheorem{proposition}[theorem]{Proposition}
\newtheorem{definition}[theorem]{Definition}

\newtheorem{example}[theorem]{Example}

\theoremstyle{definition} \theoremstyle{remark}
\numberwithin{equation}{section}

\newcommand{\iii}{\mathbf{i}}

\newcommand{\uuu}{\mathbf{u}}
\newcommand{\vvv}{\mathbf{v}}
\newcommand{\www}{\mathbf{w}}
\newcommand{\xxx}{\mathbf{x}}
\newcommand{\yyy}{\mathbf{y}}
\newcommand{\zzz}{\mathbf{z}}

\newcommand{\diam}{{\rm diam}}

\newcommand{\ed}{\mathfrak{E}}

\newcommand{\en}{\mathcal{E}}

\newcommand{\dom}{\mathcal{D}}

\newcommand{\ds}{\displaystyle}

\newcommand{\ha}{\mathcal{H}}

\begin{document}

\title{\large{{\bf CRITICAL EXPONENTS OF INDUCED DIRICHLET FORMS ON SELF-SIMILAR SETS}}}

\author{Shi-Lei Kong and Ka-Sing Lau}

\date{}

\maketitle

\abstract{In \cite{KLW}, we studied certain random walks on the hyperbolic graphs $X$ associated with the self-similar sets $K$, and showed that the discrete energy  ${\mathcal E}_X$ on $X$ has an induced energy form ${\mathcal E}_K$ on $K$ that is a Gagliardo-type integral. The domain of ${\mathcal E}_K$ is a Besov space $\Lambda^{\alpha, \beta/2}_{2,2}$ where $\alpha$ is the Hausdorff dimension of $K$ and $\beta$ is a parameter determined by the ``return ratio" of the random walk. In this paper, we study the functional relationship of ${\mathcal E}_X$ and ${\mathcal E}_K$ as well as the associated Besov spaces. In particular, we investigate the critical exponents of the $\beta$ in  $\Lambda^{\alpha, \beta/2}_{2,2}$ in order for ${\mathcal E}_K$ to be a regular Dirichlet form. We provide some criteria to determine the critical exponents through the effective resistance of the random walk on $X$, and make use of certain electrical network techniques to calculate the exponents for some concrete examples.

\renewcommand{\thefootnote}{}

\footnote {
\indent \hspace {0.5cm} reversible random walk.}
\footnote{The research is supported in part by the HKRGC grant and the NNSF of China (no. 11371382).}
\renewcommand{\thefootnote}{\arabic{footnote}}
\setcounter{footnote}{0}

\tableofcontents

\bigskip

\section{Introduction}
\label{sec:1}

Let $(K,\rho,\nu)$ be a metric measure space in which $(K, \rho)$ is a locally compact separable metric space and $\nu$ is an {\it $\alpha$-Ahlfors measure}, i.e., $\nu$ satisfies $\nu (B(x, r)) \asymp r^\alpha$ for any ball $B(x,r)$ with center at $x\in K$ and radius $r \in (0,1)$ (by $f \asymp g$, we mean $f$ and $g$ are positive functions, and $C^{-1} g\leq f \leq Cg$ for some $C>0$). We call such $K$ an $\alpha$-set in the case that $K$  is a compact subset in ${\mathbb R}^d$ with the Euclidian metric \cite{JW}.  The Besov space $\Lambda_{2,2}^{\alpha,\beta/2}$, $\beta>0$ is the Banach space contained in $L^2(K, \nu)$ defined via the norm
\begin{equation} \label{eq1.1}
 \Vert u \Vert_{{\Lambda_{2,2}^{\alpha,\beta/2}}} = \Vert u \Vert_{L^2}+ \big (\iint_{K \times K } \frac{|u(\xi)-u(\eta)|^2} {|\xi-\eta|^{\alpha+\beta}}d\nu(\xi)d\nu(\eta) \big )^{1/2},
\end{equation}
(note that $\nu\times \nu$ vanishes on the diagonal)  where the integral term is called the Gagliardo integral and denoted by ${\mathcal E}^{(\beta)}[u]$. Similarly we define another Besov space $\Lambda_{2,\infty}^{\alpha,\beta/2}$ via
 $$
 \Vert u \Vert_{{\Lambda_{2,\infty}^{\alpha,\beta/2}}}  =\Vert u \Vert_{L^2} +\Big(\sup_{0<r\leq1} r^{-(\alpha+ \beta)}\int_K\int_{B(\eta, r)} |u(\xi)-u(\eta)|^2 d\nu(\xi)d\nu(\eta)\Big)^{1/2}.
 $$
On a classical domain (with $\alpha = d)$ in ${\mathbb R}^d$, it is well-known that  $\Lambda_{2,\infty}^{\alpha,1}$ equals the Sobolev space $W^{1,2}$,  and for $0<\beta<2$, $\Lambda_{2,2}^{\alpha,\beta/2}$ equals the fractional Sobolev space $W^{s,2}$ with $s= \beta/2$ \cite{A}.
It is easy to see that $\Lambda_{2,2}^{\alpha,\beta' /2}\subset \Lambda_{2,\infty}^{\alpha,\beta'/2}  \subset \Lambda_{2,2}^{\alpha,\beta /2}$ for $\beta < \beta'$;
$\Lambda_{2,2}^{\alpha,\beta/2}$ can be trivial for sufficiently large $\beta$. We define a {\it critical exponent} $\beta^*$ of $K$ by
$$
\beta^* = \sup \{\beta>0: \Lambda_{2,2}^{\alpha,\beta/2} \hbox{ contains nonconstant functions}\}.
$$

The Besov spaces $\Lambda_{2,\infty}^{\alpha,\beta/2}$, $\Lambda_{2,2}^{\alpha,\beta/2}$ and the critical exponents play an important role in the study of the Dirichlet forms.  For a classical domain $\Omega$, the standard Dirichlet form $\mathcal E (u, v) = \int_\Omega \nabla u(x) \nabla v(x) dx$ is defined on the domain ${\mathcal D} =\Lambda_{2,\infty}^{\alpha,1} (= W^{1,2})$, and $\beta^* =2$. The theory of Dirichlet forms on a metric measure space was originated in the seminal work of Beurling and Deny \cite {BeD,FOT},  in which a local regular Dirichlet form $(\mathcal E, \mathcal D)$ (if exists) gives a generalization of Laplacian.   In \cite{Jo}, Jonsson showed that the domain $\mathcal D$ of the local regular Dirichlet form on the Sierpi{\'n}ski gasket is the Besov space $\Lambda_{2,\infty}^{\alpha,\beta^*/2}$, where $\alpha = \log 3/ \log 2$ and $\beta^* = \log 5/\log 2 (\approx 2.322$). This consideration was extended by Pietruska-Pa{\l}uba to nested fractals and $\alpha$-sets \cite {P1,P2}.  From the probabilistic point of view, the $\beta^*$ is referred to as the {\it ``walk dimension"}, which is the scaling exponent in the space-time relation of the  diffusion process (Brownian motion) $\{X_t\}$} on the underlying set $K$:  ${\mathbb E}_x(|X_t-x|^2)\approx t^{2/\beta^*}$.  Typically, $\{X_t\}$ has heat kernels that obey the sub-Gaussian estimate:
\begin{equation} \label{eq1.0}
p(t,\xi,\eta) \asymp \frac{1}{t^{\alpha/\beta^*}} \exp\Big(-c\left(\frac{|\xi-\eta|}{t^{1/\beta^*}}\right)^{\frac{\beta^*}{\beta^*-1}}\Big)
\end{equation}
(here the value of $c>0$ varies in the upper and lower bounds). In particular, Barlow and Bass in \cite{BB1,BB2,BB3} proved the above heat kernel estimate of the Brownian motion on the Sierpinski carpet, and a numerical approximation $\beta^* \approx 2.097$ is highlighted in \cite{BB2}.   The estimates \eqref{eq1.0} on nested fractals were also obtained by Kumagai \cite{Ku}, in which he evaluated $\beta^*$ for some specific cases.
 Local regular Dirichlet forms have also been studied in the general setting of metric measure spaces together with the heat kernel estimates (e.g., \cite{GH1, GHL1, GHL2, GHL4}). In particular,  Grigor'yan, Hu and Lau \cite {GHL1} proved that $2\leq  \beta^* \leq \alpha +1$ under the assumption that a sub-Gaussian heat kernel exists together with a chain condition (see also St\'{o}s \cite{St} for the same inequality  on the $\alpha$-sets). However, despite the various developments, there is no effective algorithm to determine the critical exponent $\beta^*$, and it is still an open question whether a Laplacian will exist on some more general fractal sets.

\bigskip

On a classical domain in ${\mathbb R}^d$, the Gagliardo integral ${\mathcal E}^{(\beta)}$ in \eqref{eq1.1} with $0<\beta<2$ defines a non-local regular Dirichlet form that corresponds to a fractional Laplacian and a symmetric $\beta$-stable process.  In \cite {St}, assuming a Brownian motion exists on an $\alpha$-set $K$,  St{\'o}s investigated the same type of non-local Dirichlet forms ${\mathcal E}^{(\beta)}$, $\beta <\beta^*$ from the associated stable-like processes that is subordinate to the Brownian motion, and he showed that the Besov spaces $\Lambda_{2,2}^{\alpha,\beta/2}$ are the domains of ${\mathcal E}^{(\beta)}$. For such processes, the heat kernels were studied in detail by Chen and Kumagai \cite{CK} on an $\alpha$-set with $0<\beta<2$. Recently, there is a considerable interest devoted to the regular non-local Dirichlet forms and the jump processes on metric measure spaces (e.g., \cite{CKW1,GHH,GHL3,HK}).

\bigskip

In \cite{KLW}, we studied the non-local Dirichlet forms with another approach. For a self-similar set $K$ in ${\mathbb R}^d$ with the open set condition (OSC), it is known that there is a hyperbolic graph $(X, {\mathfrak E})$ (augmented tree) on the symbolic space $X$ of $K$, and the hyperbolic boundary and $K$ are H\"older equivalent\cite{Ka, LW1, LW2}. On $(X, {\mathfrak E})$, we introduced a class of transient reversible random walks with return ratio $\lambda \in (0, 1)$ (the conductance $c(\bf x, \bf y)$ depends on $\lambda$), and called them {\it $\lambda$-natural random walks }($\lambda$-NRW) (see Section \ref{sec:2}). The random walk was shown to satisfy the conditions of Ancona's theorem in \cite{An} so that the Martin boundary and the hyperbolic boundary (and hence $K$) are homeomorphic. Moreover, the hitting distribution $\nu$ is the normalized $\alpha$-Hausdorff measure where $\alpha$ is the Hausdorff dimension of $K$. By using a boundary theory of Silverstein on Markov chains \cite{Si}, we proved that the graph energy
  \begin{equation} \label{eq1.2}
  {\mathcal E}^{(\lambda)}_X [f] = \frac 12 {\sum}_{\xxx, \yyy \in X} c(\xxx,\yyy)|f(\xxx) -f(\yyy)|^2
  \end{equation}
  defined by the  $\lambda$-NRW induces a  non-negative definite bilinear form on $K$:
\begin{equation} \label {eq1.3}
\mathcal E_K^{(\beta)}(u,v) \asymp \iint_{K \times K } \frac{(u(\xi)-u(\eta))(v(\xi)-v(\eta))}{|\xi-\eta|^{\alpha+\beta}} d\nu(\xi)d\nu(\eta)
\end{equation}
with $\beta = \log \lambda/ \log r$, where $r$ equals the minimal contraction ratio among the maps in the IFS that generates $K$. Clearly the domain ${\mathcal D}^{(\beta)}_K = \{u \in L^2(K, \mu): {\mathcal E}^{(\beta)}_K [u]< \infty\}$ is the Besov space $\Lambda_{2,2}^{\alpha,\beta/2}$.

\medskip

As we see from the above, unlike the classical case, the  Dirichlet forms in \eqref{eq1.3} can be obtained more directly on general self-similar sets without recourse to the local regular Dirichlet form (i.e., the Laplacian). In this paper we continue the investigation of the induced bilinear functional $\mathcal E_K^{(\beta)}$. We are aiming for a deeper understanding of the boundary theory of this $\lambda$-NRW, in particular on the critical exponents, so as to shed some light on the problem of the existence of Laplacians on the more general fractal sets. We will focus on two issues, namely, to establish the functional relationship of the discrete energy ${\mathcal E}_X^{(\lambda)}$ and the induced ${\mathcal E}_K^{(\beta)}$ where $\beta = \log \lambda/ \log r$, then use it to study the critical exponents of  $\{\Lambda^{\alpha, \beta/2}_{2,2}\}_{\beta>0}$, domains of $\mathcal E_K^{(\beta)}$'s.

 \medskip

 Let ${\mathcal D}_X^{(\lambda)}$ be the domain of ${\mathcal E}_X^{(\lambda)}$, and let ${\mathcal H}{\mathcal D}^{(\lambda)}_X$ be the class of harmonic functions in  ${\mathcal D}_X^{(\lambda)}$. For $u \in {\mathcal D}_K^{(\beta)}$, we use $Hu$ to denote the Poisson integral of $u$ on $X$,  and for  $f \in {\mathcal D}_X^{(\lambda)}$, we let ${\rm Tr}f (\xi) = \lim_{\xxx_n \to \xi} f(\xxx_n)$. By imposing a norm on ${\mathcal D}_X^{(\lambda)}$,  we prove a theorem analogous to the classical trace theorem (see Theorem \ref{th3.6}, Corollaries \ref{th3.6'} and \ref{th3.5}).

\vspace{1mm}

\begin{theorem} \label{th1.1} Suppose $K$ is a self-similar set and assume that the OSC holds. Then for a $\lambda$-NRW with $\lambda \in (0, r^\alpha)$, ${\rm Tr}({\mathcal D}_X^{(\lambda)}) = {\mathcal D}_K^{(\beta)}$. Moreover, ${\rm Tr}: \mathcal {HD}^{(\lambda)}_X \to {\mathcal D}^{(\beta)}_K$ is a Banach space isomorphism, and ${\rm Tr}^{-1} = H$  on ${\mathcal D}^{(\beta)}_K$. (Here $\beta = \log \lambda / \log r$.)
\end{theorem}

\medskip

The condition  $\lambda \in (0, r^\alpha)$ in Theorem \ref{th1.1} will be used throughout the paper.  It implies that $\beta > \alpha$, and functions in ${\mathcal D}^{(\beta)}_K$ are H\"older continuous (Proposition \ref{th2.4}); moreover, the convergence rate  $(\lambda/r^\alpha)^n$ is essential when we consider functions in ${\mathcal D}^{(\lambda)}_X$ that tend to the boundary $K$.

\medskip

To consider the critical exponent of ${\mathcal D}^{(\beta)}_K$,   we introduce some finer classification of the domains. We let
\begin{align*}
\beta_1^*:= & \sup \{\beta > 0: {\mathcal D}^{(\beta)}_K \cap C(K) \hbox{ is dense in } C(K)\}, \\
\beta_2^*:=& \sup \{\beta>0: \dim {\mathcal D}^{(\beta)}_K = \infty\} \\
\beta_3^*:= & \sup \{\beta > 0: {\mathcal D}^{(\beta)}_K \hbox{ contains nonconstant functions} \},
\end{align*}
Clearly we have $2 \leq \beta_1^* \leq \beta_2^* \leq \beta_3^* \leq \infty$, and $\beta_3^* = \beta^*$ for the $\beta^*$ defined previously.  In the standard cases, these three exponents are equal, but there are also examples that they are different \cite {GuL}.  We will discuss these exponents and to provide some criteria to determine them. Our approach relies on the effective resistance. We use $R^{(\lambda)}(\xi, \eta)$ to denote the {\it limiting resistance} for $\xi, \eta \in K$ (see Section \ref{sec:4}), and note that the infinite word $i^\infty$ of $\{S_i\}_{i=1}^N$ will represent an element in $K$.

\medskip

\begin{theorem} \label{th1.2}  With the assumptions as in Theorem \ref{th1.1}, the domain ${\mathcal D}^{(\beta)}_K$ consists of only constant functions if and only if $R^{(\lambda)}(i^\infty, j^\infty) =0 $ for all $i, j = 1,\cdots, N$.

\vspace {0.1cm}

Consequently, (i) if we let $\lambda^*_3 = \sup \{\lambda>0: \ R^{(\lambda)}(i^\infty, j^\infty) =0,  \   1\leq i,j \leq N\}$, then $\beta^*_3 = \log  \lambda^*_3/ \log r$; \,(ii) if $\beta_3^* > \alpha$ and $K$ is connected, then $\beta_2^* = \beta_3^*$ .
\end{theorem}

\medskip
The theorem is proved in Theorems \ref{th4.4} and \ref{th4.6}. The main idea is that  the condition $R^{(\lambda)}(\cdot,\cdot)=0$ on   the finite set $\{i^\infty\}$ implies that it equals zero on a dense subset in $K$, and this leads to an infinite effective conductance on the dense subset. Then the continuity of $u \in {\mathcal D}^{(\beta)}_K$ implies that $u$ can only be a constant function.

\medskip

For $\beta_1^*$, we have a result on the post critically-finite (p.c.f.) sets \cite{Ki1}. We let $V_0$ denote the ``boundary" of $K$.

\medskip

\begin{theorem} \label{th1.3}
If in addition, $K$ is a p.c.f.~set and satisfies another mild geometric condition (see Theorem \ref{th4.9}). Then if
$$
R^{(\lambda-\epsilon)}(\xi, \eta)>0, \qquad \forall \ \xi \neq \eta \in V_0,
$$
for some $0<\epsilon <\lambda$, then $\dom_K^{(\beta)}$ is dense in $C(K)$ with supremum norm.

\vspace {0.1cm}
Consequently, if
$
\lambda_1^*:= \inf \{\lambda>0: R^{(\lambda)}(\xi, \eta)>0, \ \forall \, \xi \neq \eta \in V_0\} \in (0, r^\alpha),
$
then $\beta_1^* = \log \lambda_1^* / \log r$.
\end{theorem}

\medskip

A challenging task is to determine the limiting resistance $R^{(\lambda)}(i^\infty,j^\infty)$ (or $R^{(\lambda-\epsilon)}(\xi, \eta)$ for $\xi, \eta \in V_0$) to be $=0$ or $>0$ in the above theorems. For this we make use of the basic tools in the electrical network theory (series and parallel laws, $\Delta$-Y transform, as well as cutting and shorting) for such estimation. We provide some special cases as examples.

\bigskip

For the organization of the paper, in Section \ref{sec:2}, we summarize the needed results  from \cite{KLW}. In Section \ref{sec:3}, we prove some basic results on the limits of functions in ${\mathcal D}^{(\lambda)}_X$ as well as the extension of functions in ${\mathcal D}^{(\beta)}_K$ via the Poisson integral,  and prove Theorem \ref{th1.1}. We define and justify the limiting  resistance in Section \ref{sec:4}, and prove Theorems \ref{th1.2} and \ref{th1.3} in Section \ref{sec:5}. In Section \ref{sec:6}, we make use of the electrical techniques to give some implementations of the theorems by some examples. Some remarks and open problems are provided in Section \ref{sec:7}.

\bigskip

\section{Preliminaries}
\label{sec:2}

\noindent  We will give a brief summary of the background results in \cite{KLW} for the convenience of the reader, and all the unexplained notations can be found there. Let $\{S_i\}_{i=1}^N$, $N\geq 2$, be an {\it iterated function system} (IFS) of contractive similitudes on ${\mathbb R}^d$ with contraction ratios $\{r_i\}_{i=1}^N$, and let $K$ be the {\it self-similar set}.  Let $\Sigma^*$ be the symbolic space of $K$. Let $r = \min \{r_i : i = 1, \cdots, N\}$. For $n\geq 1$, define
\begin{equation}
\mathcal{J}_n = \{\xxx = i_1\cdots i_k \in \Sigma^\ast: r_\xxx \leq r^n < r_{i_1 \cdots i_{k-1}}\},
\end{equation}
and $\mathcal{J}_0 = \{\vartheta\}$ by convention.  Consider the {\it modified symbolic space}  $X = \bigcup_{n=0}^\infty\mathcal{J}_n$, which has a tree structure with a set ${\mathfrak E}_v$ of vertical edges.  The tree can be strengthened to a more structural hyperbolic graph by adding  horizontal edges  according to the neighboring cells on each level $n$ \cite{Ka, LW1, LW2}.  According to \cite{LW2}, we define
$$
{\mathfrak E}_h = \bigcup_{n=0}^\infty \{(\xxx, \yyy) \in \mathcal{J}_n \times  \mathcal{J}_n: \xxx \neq \yyy, \, \inf_{\xi,\eta \in K} |S_\xxx(\xi)-S_\yyy(\eta)| \leq \gamma \cdot r^n\},
$$
where $\gamma>0$ is arbitrary but fixed. Let ${\mathfrak E} = {\mathfrak E}_v \cup {\mathfrak E}_h$, and call $(X, {\mathfrak E})$ an {\it augmented tree}, coined  by Kaimanovich in \cite{Ka}.
It was shown that $(X, \mathfrak E)$ is a hyperbolic graph in the sense of Gromov \cite {Wo1}. In this case, the lengths of horizontal geodesics are uniformly bounded, and for any $\xxx, \yyy \in X$, the canonical geodesic $[\xxx, \uuu, \vvv, \yyy]$ consists of three segments, where $[\xxx, \uuu], [\vvv, \yyy]$ are vertical paths in ${\mathfrak E}_v$, and $[\uuu, \vvv]$ is a horizontal geodesic in ${\mathcal J}_\ell$ with the smallest $\ell$.  Using this geodesic, the {\it Gromov product} $(\xxx | \yyy)$ has a simple and useful geometric interpretation:
$$
(\xxx | \yyy) = \ell - h/2,
 $$
where $h$ is the length of $[\uuu, \vvv]$ and $h$ is uniformly bounded. For some $a>0$, there is a {\it Gromov metric} $\rho_a$ on $X$ such that $\rho_a (\xxx, \yyy) \asymp e^{-a(\xxx | \yyy)}$ for all $\xxx \neq \yyy$. Let ${\widehat X}_H$ be the completion of $(X,\rho_a)$, and define the {\it hyperbolic boundary} $\partial_HX = {\widehat X}_H \setminus X$. Then $(\partial_HX,\rho_a)$ is a compact metric space.

\medskip
A {\it geodesic ray}  $(\xxx_n)_{n=0}^\infty $ is a sequence of words with $\xxx_n = i_1 i_2 \cdots i_{k(n)} \in {\mathcal J}_n$. If $\xi \in \partial_HX$ has a canonical representation $i_1i_2 \cdots \in \Sigma^\infty$, then  $(\xxx_n)_n$ converges to $\xi$, and $\xi \in S_{\xxx_n} (K)$ for all $n$. It follows that for any other geodesic ray $(\yyy_n)_n$ converging to $\xi$, we have $\xxx_n \sim_h \yyy_n$.
In the sequel,  we will make use of the geodesic rays frequently to relate functions on $X$ and $K$.  We call the sequence $\{\kappa_n\}_{n=0}^\infty$ a {\it $\kappa$-sequence} if each $\kappa_n$ is a selection map from $K$ to $\mathcal J_n$, such that for each $\xi \in K$, $(\kappa_n(\xi))_{n=0}^\infty$ is a geodesic ray converging to $\xi$.
It follows from the above that

\medskip
\begin{lemma} \label{th2.0} For any IFS $\{S_i\}_{i=1}^N$, let $(X, {\mathfrak E})$ be the hyperbolic graph as defined above. Let $E$ be a closed subset $K$. Then for any two $\kappa$-sequences $\{\kappa_n\}_{n=0}^\infty$ and $\{\kappa'_n\}_{n=0}^\infty$, we have
$$
\kappa'_n (E) \subset \{\xxx \in {\mathcal J}_n: d(\xxx, \kappa_n(E))\leq 1\} \quad \hbox {for each}\  n ,
$$
where $d(\cdot, \cdot)$ is the graph metric on $(X,\mathfrak E)$.
\end{lemma}
\medskip

\begin{theorem} \label{th2.1} \hspace{-2mm} {\rm \cite{Ka, LW2}} For any IFS $\{S_i\}_{i=1}^N$, let $(X, {\mathfrak E})$ be the hyperbolic graph as defined above. Then  the hyperbolic boundary is H\"older equivalent to the self-similar set $K$, i.e., for the canonical map $\iota: \partial_HX \to K$,
$$
\rho_a(\xi, \eta) (\asymp e^{-a(\xxx | \yyy)}) \asymp |\iota (\xi) - \iota(\eta)|^{- a/\log r}.
$$
\end{theorem}

\medskip

Throughout this paper, we will always assume that the IFS $\{S_i\}_{i=1}^N$ satisfies the {\it open set condition} (OSC) \cite{Fa}. In this case, the self-similar set $K$ has Hausdorff dimension $\alpha$ which is uniquely determined by  ${\sum}_{i=1}^N r_i^\alpha = 1$.

\medskip

In \cite{KLW}, we introduced a class of reversible random walks on the augmented tree $(X,\mathfrak E)$: for $\lambda \in (0,1)$, we set the {\it conductance} $c: \mathfrak E \to (0,\infty)$ such that
\begin{equation} \label{eq2.1}
c(\xxx,\xxx^-) = r_\xxx^\alpha \lambda^{-|\xxx|}, \quad \hbox {and} \ \   c(\xxx,\yyy) \asymp r_\xxx^\alpha \lambda^{-|\xxx|}, \ \  \xxx \sim_h \yyy \in X \setminus \{\vartheta\},
\end{equation}
where $\xxx^-$ is the parent of $x$, $r_\xxx := r_{i_1} \cdots r_{i_m}$ for $\xxx=i_1 \cdots i_m$. (For example, for the Sierpinski gasket, $r^\alpha = 1/3$, and $ c(\xxx,\xxx^-)= (3\lambda)^{-|\xxx|}$.) We define the {\it natural random walk} with return ratio $\lambda \in (0,1)$ ($\lambda$-NRW) to be the Markov chain $\{Z_n\}_{n=0}^\infty$ on $X$ with transition probability
$P(\xxx,\yyy)=c(\xxx,\yyy) / m(\xxx)$ if $\xxx \sim \yyy$, and $0$ otherwise,
 where $m(\xxx)=\sum_{\yyy: \xxx \sim \yyy} c(\xxx,\yyy)$ is the {\it total conductance} at $\xxx \in X$.
Note that the random walk has a return ratio $\lambda \in (0, 1)$  with respect to the vertical direction; hence $\{Z_n\}_{n=0}^\infty$ is transient. Let ${\mathcal M}$ denote the Martin boundary, and let $Z_\infty$ be the $\mathcal M$-valued random variable as the limit of $\{Z_n\}_{n=0}^\infty$.

\begin{theorem} \label{th2.2} \hspace{-2mm} {\rm \cite{KLW}} Let $\{S_i\}_{i=1}^N$  be an IFS satisfying the open set condition, and let $\{Z_n\}_{n=0}^\infty$ be a $\lambda$-NRW. Then
\begin{enumerate}
\item[(i)] the Martin boundary ${\mathcal M}$, the hyperbolic boundary $\partial_HX$ and the self-similar set $K$ are all homeomorphic;

\item[(ii)] the Martin kernel $K(\xxx, \xi) \asymp \lambda^{|\xxx| -(\xxx|\xi)} r^{-\alpha(\xxx|\xi)}$;

\item[(iii)] the distribution $\nu$ of $Z_\infty$ on ${\mathcal M}$ equals the normalized $\alpha$-Hausdorff measure on $K$ when $Z_0 = \vartheta$.
\end{enumerate}
\end{theorem}

\medskip

We will fix $\lambda \in (0,1)$, and when there is no confusion,  we will omit the superscripts of $\lambda$ and $\beta (:=\log \lambda/\log r)$ in the involved notations on $X$ and $K$.  It follows from part (i) that we can carry  Doob's discrete potential theory onto the self-similar set $K$. We denote the space of {\it harmonic} functions (w.r.t.~$P$) on $X$ by $\ha(X) = \{f \in \ell(X): Pf= f\}$, where $\ell(X)$ is the collection of real functions on $X$, and $Pf(\xxx) = \sum_{\yyy \in X} P(\xxx,\yyy)f(\yyy)$. The Poisson integral for $u \in L^1 (K, \nu)$ is
\begin{equation} \label{eq2.2'}
Hu(\cdot) = \int_K K(\cdot, \xi)u(\xi)d\nu(\xi) \in \ha(X).
\end{equation}
The {\it graph energy} of $f \in \ell(X)$ is given by
\begin{equation} \label{eq2.2}
\en_X[f] = \dfrac{1}{2}\sum_{\xxx,\yyy \in X: \xxx \sim \yyy} c(\xxx,\yyy)(f(\xxx)-f(\yyy))^2,
\end{equation}
and the domain of $\en_X$ is $\dom_X = \{f \in \ell(X): \en_X[f] < \infty\}$. Using Theorem \ref{th2.2} together with Silverstein's approach on the Na{\"i}m kernel $\Theta(\xi,\eta)$ on $K$ \cite{Si}, we obtain an induced quadratic form on $K$ as follows.

\medskip
\begin{theorem} \label{th2.3} \hspace{-2mm} {\rm \cite {KLW}} Under the assumptions in Theorem \ref{th2.2}, the graph energy in \eqref{eq2.2}  induces an energy form $\en_K[u]:= \en_X[Hu] $ given by
\begin{equation} \label{eq2.3}
\en_K[u] =\frac{m(\vartheta)}{2}\iint_{K \times K} |u(\xi)-u(\eta)|^2 \Theta(\xi,\eta)d\nu(\xi)d\nu(\eta), \quad u \in L^2(K,\nu),
\end{equation}
where $\Theta(\xi,\eta) \asymp (\lambda r^\alpha)^{-(\xi|\eta)} \asymp |\xi-\eta|^{-(\alpha+\beta)}$ with $\beta = \frac{\log \lambda}{\log r}$.
\end{theorem}

\bigskip
The domain of $\en_K$ is  $\dom_K = \{u \in L^2(K,\nu): Hu \in \dom_X\}$ \cite{Si}.
It follows from $\en_K[u]:= \en_X[Hu]$ that $\dom_K$ also equals  $\{u \in L^2(K,\nu): {\mathcal E}_K(u) < \infty \}$.
Hence ${\mathcal D}_K$ is the Besov space ${\Lambda_{2,2}^{\alpha,\beta/2}}$. If we define $\Vert u\Vert_{{\mathcal E}_K}^2 = {\mathcal E}_K [u] + \Vert u \Vert^2_{L^2(K, \nu)}$, then $({\mathcal D}_K, \Vert \cdot\Vert_{{\mathcal E}_K})$ is a Banach space, and is equivalent to ${\Lambda_{2,2}^{\alpha,\beta/2}}$.
 For $\gamma >0$, we let
\begin{equation} \label{eq2.8}
C^\gamma (K)=\{u \in C(K): \Vert u \Vert_{C^\gamma}:=\Vert u \Vert_\infty + {\rm esssup}_{\xi,\eta \in K} \frac{|u(\xi)-u(\eta)|}{|\xi-\eta|^\gamma}<\infty\}
\end{equation}
denote the H\"older space. We will use the following result frequently. It was proved in {\rm \cite{GHL1}} (the assumption of heat kernel stated there is not needed in the proof) that

\medskip

\begin{proposition} \label{th2.4}
If $\beta>\alpha$, then for all $u \in L^2(K,\nu)$,
\begin{equation} \label{eq2.9}
\Vert u \Vert_{C^\gamma} \leq C\Vert u \Vert_{\Lambda_{2,2}^{\alpha,\beta/2}}
\end{equation}
with $\gamma = (\beta-\alpha)/2$. Consequently, $\Lambda_{2,2}^{\alpha,\beta/2} \hookrightarrow C^\gamma$ is an imbedding.
\end{proposition}

\bigskip

It follows that for $\alpha < \beta <\beta^*_1$, ${\mathcal D}_K \cap C(K)= {\mathcal D}_K $ is trivially dense in ${\mathcal D}_K$  under the norm $\Vert \cdot \Vert_{{\mathcal E}_K}$, and in $C(K)$ under the supremum norm. This implies that $({\mathcal E}_K, {\mathcal D}_K)$ is a non-local regular  Dirichlet form.

\section{Harmonic functions and trace functions}
\label{sec:3}

\noindent In this section, we will set up a natural relation between the finite energy  harmonic functions on $X$  and  the finite induced energy functions on $K$ (Theorem \ref{th3.6}).
First we use Theorem \ref{th2.2}(ii) to provide a ``uniform tail estimate" of the Martin kernel. As in \cite[Section 5]{KLW}, we introduce a projection $\iota: X \to K$ by selecting $\iota(\xxx) \in S_\xxx(O \cap K)$ arbitrarily, where $O$ is an open set in the OSC satisfying $O \cap K \neq \emptyset$.

\medskip

\begin{proposition}\label{th3.1}
Let $\{S_i\}_{i=1}^N$ be an IFS satisfying the OSC, and let $\{Z_n\}_{n=0}^\infty$ be a $\lambda$-NRW on the augmented tree $(X,\mathfrak E)$. Then for any $\epsilon, \delta>0$, there exists a positive integer $n_0$ such that for any $\xxx \in X$ and $|\xxx| \geq n_0$, $K(\xxx, \xi) \leq \varepsilon$
 for any $\xi \in K \setminus B(\iota(\xxx), \delta)$.
\end{proposition}

\begin{proof}
It follows from Theorem \ref{th2.2}(ii) that
$$
K(\xxx, \xi) \leq C_1\lambda^{|\xxx|}(\lambda r^\alpha)^{-(\xxx|\xi)},  \quad  \xxx \in X, \ \xi \in K.
$$
Note that $(\xxx|\xi) \leq (\iota(\xxx)|\xi)$ by \cite[Lemma 3.7(ii)]{KLW}. Hence for $\xi \in K \setminus B(\iota(\xxx), \delta)$,
$$
r^{-(\xxx|\xi)} \leq r^{-(\iota(\xxx)|\xi)} \leq C_2 |\iota(\xxx)-\xi|^{-1} \leq C_2 \delta^{-1}
$$
(the second inequality follows from Theorem \ref{th2.1}).
Hence for $\varepsilon>0$, we can pick a large integer $n_0$ such that the last inequality in the following holds:
$$
K(\xxx, \xi) \leq C_1 \lambda^{n_0} r^{-(\alpha+\log \lambda / \log r)(\xxx | \xi)} \leq C_1\lambda^{n_0} (C_2 \delta^{-1})^{\alpha+\log \lambda / \log r} \leq \varepsilon.
$$
\end{proof}

\medskip

Let $\nu_\xxx$, $\xxx \in X$, denote the hitting distribution of $Z_\infty$ on $K$, starting from $\xxx$. As $K(\xxx, \cdot) = d\nu_\xxx / d\nu$, the above result shows that the mass of $\nu_\xxx$ will concentrate around $\iota(\xxx)$ (equivalently, $S_{\xxx}(K)$) as $|\xxx| \to \infty$.
We have a Fatou-type theorem as a corollary.

\medskip

\begin{corollary} \label{th3.2}
Suppose $\{S_i\}_{i=1}^N$ satisfies OSC, and let $\{Z_n\}_{n=0}^\infty$ be a $\lambda$-NRW on the augmented tree $(X,\mathfrak E)$. Then for $u \in C(K)$ and $\varepsilon>0$, there exists a positive integer $n_0$ such that
\begin{equation} \label{eq3.1}
|Hu(\xxx)-u(\xi)| \leq \varepsilon, \qquad   \forall \ |\xxx| \geq n_0 ,  \ \xi \in S_\xxx(K).
\end{equation}
 In particular,  $\lim_{n\to \infty } Hu(\xxx_n) = u(\xi)$ uniformly for $\xi \in K$, where $(\xxx_n)_n$ is a geodesic ray converging to $\xi$.
\end{corollary}

\begin{proof}
Since $u$ is continuous on the compact set $K$, $u$ is bounded and uniformly continuous. We let $\sup_{\xi \in K} |u(\xi)| = M_0 < \infty$ and choose $\delta>0$ such that $|u(\xi)-u(\eta)|<\varepsilon/3$ whenever $|\xi-\eta|<\delta$ on $K$. Furthermore, by Proposition \ref{th3.1}, we choose $n_0$ such that both $\diam(S_\xxx(K)) \leq \delta$ and $K(\xxx,\xi)\leq \frac {\varepsilon}  {6M_0}$ hold for any $\xxx \in X$ with $|\xxx| \geq n_0$ and $\xi \in K \setminus B(\iota(\xxx), \delta)$. Then for $|\xxx| \geq n_0$, by using the usual technique of splitting the following integral on $K$ into $K \cap B(\iota(\xxx), \delta)$ and $K \setminus B(\iota(\xxx), \delta)$, we can show that
\begin{align*}
|Hu(\xxx)-u(\iota(\xxx))| \leq \ds\int_K |K(\xxx,\eta)(u(\eta)-u(\iota(\xxx)))| d\nu(\eta) \leq \varepsilon
\end{align*}
Hence for $\xi \in S_\xxx(K)$,
$$
|Hu(\xxx)-u(\xi)| \leq |Hu(\xxx)-u(\iota(\xxx))|+|u(\iota(\xxx))-u(\xi)| \leq \dfrac{2\varepsilon}{3}+\dfrac{\varepsilon}{3} = \varepsilon,
$$
and \eqref{eq3.1} holds. For the last statement, let $(\xxx_n)_n$ be a geodesic ray converging to $\xi$,  then $\xxx_n = i_1 \cdots i_n$, and this $i_1i_2 \cdots \in \Sigma^\infty$ is a representation of some $\xi'$ with $\xi' \in S_{\xxx_n} (K)$, and $\xi'=\xi$ in $\partial_HX$.  Hence  by \eqref{eq3.1}, we have $\lim_{n\to \infty} Hu(\xxx_n) = u(\xi') =u(\xi)$, and the convergence is uniform on $\xi$.
\end{proof}

\bigskip

In the rest of this section, we assume that the $\lambda$-NRW has a return ratio $\lambda \in (0,r^\alpha)$. Then $\beta  = \log \lambda/\log r>\alpha$, and Proposition \ref{th2.4} applies.

\medskip

\begin{lemma} \label{th3.3}
Suppose $\{S_i\}_{i=1}^N$ satisfies OSC, and let $\{Z_n\}_{n=0}^\infty$ be a $\lambda$-NRW on the augmented tree $(X,\mathfrak E)$ with $\lambda \in (0,r^\alpha)$. Then for $f \in \dom_X$,
\begin{enumerate}
\item[(i)] there exists $C>0$ (depend on $f$)  such that for any geodesic ray $(\xxx_n)_n$, $$|f(\xxx_{n+1}) - f(\xxx_n)| \leq C (\lambda/r^\alpha)^{n/2},$$ and hence $\lim\limits_{n \to \infty} f(\xxx_n)$ exists;

\item[(ii)] for two equivalent geodesic rays $(\xxx_n)_n$ and $(\yyy_n)_n$, $\lim\limits_{n \to \infty} f(\xxx_n) = \lim\limits_{n \to \infty} f(\yyy_n)$.
\end{enumerate}
\end{lemma}

\begin{proof}
(i)  Let $\tau = \lambda / r^\alpha <1$. For a geodesic ray $(\xxx_n)_n$, since
\begin{equation}  \label{eq3.1'}
|f(\xxx_{n+1})-f(\xxx_n)| \leq \sqrt{\dfrac{\en_X[f]}{c(\xxx_{n+1},\xxx_n)}} \leq C(\lambda/r^\alpha)^{n/2} = C\tau^{n/2},
\end{equation}
hence the sequence $(f(\xxx_n))_n$ converges in an exponential rate.

\vspace{0.1cm}

 (ii) For two equivalent geodesic rays $(\xxx_n)_n$ and $(\yyy_n)_n$ that converge to the same $\xi$, if they are distinct, then $\xxx_n \sim_h \yyy_n$ for all $n$ (or by Lemma \ref{th2.0}). Then
\begin{equation}\label{eq3.2}
|f(\xxx_n)-f(\yyy_n)| \leq \sqrt{\dfrac{\en_X[f]}{c(\xxx_n,\yyy_n)}} \leq C'\tau^{n/2},
\end{equation}
which tends to 0 as $n \to \infty$. Hence the two limits are equal.
\end{proof}

\medskip

With the assumption as in Lemma \ref{th3.3}, we can define a linear map  ${\rm Tr}: \dom_X \rightarrow \ell(K)$ (called it a {\it trace map}) by
\begin{equation} \label{eq3.3}
({\rm Tr}f)(\xi) = \lim\limits_{n \to \infty} f(\xxx_n), \qquad \xi \in K,
\end{equation}
where $(\xxx_n)_n$ is a geodesic ray that converges to $\xi$. We call ${\rm Tr}f$ the {\it trace function} of $f$. By Lemma \ref{th3.3}(ii), the limit in \eqref{eq3.3} is ``uniform" in the sense that
for $f \in \dom_X$ and $\varepsilon>0$, there exists a positive integer $n_0$ such that
\begin{equation} \label{eq3.4}
|f(\xxx)-{\rm Tr}f(\xi)| \leq \varepsilon,  \qquad \forall \ |\xxx| \geq n_0, \ \xi \in S_\xxx(K).
\end{equation}

\medskip

\begin{lemma} \label{th3.4}
Suppose $\{S_i\}_{i=1}^N$ satisfies OSC, and let $\{Z_n\}_{n=0}^\infty$ be a $\lambda$-NRW with ratio $\lambda \in (0,r^\alpha)$ on the augmented tree $(X,\mathfrak E)$. Then ${\rm Tr} f$ is continuous on $K$.
\end{lemma}

\begin{proof} For $\varepsilon>0$, by \eqref{eq3.4}, there exists  $n_0$ such that $|f(\xxx)-{\rm Tr}f(\xi)|<\varepsilon/3$ for $|\xxx| \geq n_0$ and $\xi \in S_\xxx(K)$. Let $M$ be the uniform bound of the  horizontal geodesics in $(X, {\mathfrak E})$  \cite {LW1},  and let $C$ be a constant such that $c(\xxx,\yyy) \geq C^{-1}(r^\alpha/\lambda)^{|\xxx|}$ for all $\xxx \sim_h \yyy$. By assumption $\tau := \lambda / r^\alpha <1$.  We choose $n_1 \geq n_0$ such that $M\sqrt{C\en_X[f]\tau^{n_1}}<\varepsilon/3$.

\vspace{0.1cm}

As $|\xi-\eta| \asymp r^{(\xi|\eta)}$ (Theorem \ref{th2.1}), we can pick $\delta > 0$ such that $(\xi|\eta) \geq n_1$ whenever $|\xi-\eta|<\delta$. Now for $\xi,\eta \in K$ with $|\xi-\eta|<\delta$, consider a canonical geodesic $[\xi, \uuu, \vvv, \eta]$ with horizontal geodesic $(\uuu=\uuu_0,\uuu_1,\ldots,\uuu_k=\vvv)$ (see Section \ref{sec:2}). Then $|\uuu| \geq (\xi|\eta) \geq n_1$, and hence
\begin{align*}
|{\rm Tr}f(\xi)-{\rm Tr}f(\eta)| &\leq |{\rm Tr}f(\xi)-f(\uuu)|+|f(\uuu)-f(\vvv)|+|f(\vvv)-{\rm Tr}f(\eta)| \\
&< \dfrac{\varepsilon}{3}+\sum_{i=0}^{k-1}|f(\uuu_i)-f(\uuu_{i+1})|+\dfrac{\varepsilon}{3} \\
&< \dfrac{2\varepsilon}{3}+M\sqrt{C\en_X[f]\tau^{n_1}} < \varepsilon. \qquad \hbox {(by \eqref{eq3.2})}
\end{align*}
This concludes that ${\rm Tr}f \in C(K)$.
\end{proof}

\medskip

\begin{theorem} \label{th3.6}
Suppose $\{S_i\}_{i=1}^N$ satisfies the OSC, and let $\{Z_n\}_{n=0}^\infty$ be a $\lambda$-NRW with ratio $\lambda \in (0,r^\alpha)$ on the augmented tree $(X,\mathfrak E)$. Then ${\rm Tr}({\mathcal H}{\mathcal D}_X ) = \dom_K$ where ${\mathcal H}{\mathcal D}_X$ is the class of harmonic functions in $\dom_X$. More precisely,
${\rm Tr}Hu=u$ for $u \in \dom_K$, and $H{\rm Tr}f=f$ for $f \in {\mathcal H}{\mathcal D}_X$.
\end{theorem}

\begin{proof}
For $u \in \dom_K$, by definition we have $Hu \in {\mathcal H}{\mathcal D}_X$. Note that $\dom_K \cap C(K) = \dom_K$, as $\dom_K = \Lambda^{\alpha, \beta/2}_{2,2}$  can be imbedded into the H\"{o}lder space $C^{(\beta-\alpha)/2}(K)$ if $\beta>\alpha$ (Proposition \ref{th2.4}). By Corollary \ref{th3.2}, we have ${\rm Tr}Hu=u$.

\vspace{0.2cm}

 For $f \in {\mathcal H}{\mathcal D}_X$, let $u = {\rm Tr}f$.  Then $u \in C(K)$ (Lemma \ref{th3.4}). For any $\varepsilon>0$, by \eqref{eq3.4} and Corollary \ref{th3.2}, there exists a positive integer $n_0$ such that for $|\xxx| \geq n_0$ and $\xi \in S_\xxx(K)$,
\begin{equation} \label{eq3.5}
|f(\xxx)-u(\xi)|<\dfrac{\varepsilon}{2} \quad \hbox{and} \quad |Hu(\xxx)-u(\xi)|<\dfrac{\varepsilon}{2}.
\end{equation}
We show that $f = Hu$ on $X$. Suppose otherwise, we can assume without loss of generality that $f(\xxx_0)>Hu(\xxx_0)$ for some $\xxx_0 \in \mathcal J_m$. Let $a_n = \max_{\xxx \in \mathcal J_n}(f(\xxx)-Hu(\xxx))$, $n \geq 1$. Note that $f-Hu$ is harmonic. By the maximum principle of harmonic functions, we regard $\mathcal J_{n+1}$ as the boundary of $X_{n+1}=\bigcup_{k=0}^{n+1} \mathcal J_k$. Then $a_{n+1} \geq \max_{\xxx \in X_n} (f(\xxx)-Hu(\xxx)) = a_n$, thus the sequence $\{a_n\}$ is non-decreasing. Hence $\inf_{n \geq m} a_n = a_m > 0$. This contradicts that $\lim_{n \to \infty} a_n = 0$ by \eqref{eq3.5}. We conclude that $f = Hu = H{\rm Tr}f$.
\end{proof}

\medskip

Let $\vartheta$ be the root of $(X, \mathfrak E)$, then ${\mathcal D}_X$ is a Hilbert space under the inner product $\langle f, g\rangle_\vartheta = f(\vartheta) g(\vartheta) + {\mathcal E}_X (f, g)$.  Let $||\cdot ||_\vartheta$ denote the norm, and let ${\mathcal D}_{X,0}$ be the $||\cdot ||_\vartheta$-closure of  functions on $X$ with finite supports.  It is known that for $f \in {\mathcal D}_X$, it admits a decomposition $f = f_{\mathcal H} + f_0$ where $f_{\mathcal H} \in {\mathcal H}{\mathcal D}_X$ and $f_0\in {\mathcal D}_{X,0}$ \cite [Theorem 3.69] {So}.

\medskip

\begin {corollary} \label {th3.6'} With the same assumption as in Theorem \ref{th3.6}, then for $f \in {\mathcal D}_X$, we have  ${\rm Tr}f  ={\rm Tr}f_{\mathcal H}$, and hence ${\rm Tr} f \in {\mathcal D}_K$.
\end{corollary}

\begin {proof}
It suffices to show that $ {\rm Tr}f \equiv 0$  for $f \in {\mathcal D}_{X, 0}$, then the corollary follows from the above decomposition  and Theorem \ref{th3.6}  that ${\rm Tr}({\mathcal H}{\mathcal D}_X) = {\mathcal D}_K$.

First we claim  that if $\{g_\ell\}_\ell \subset {\mathcal D}_X$ satisfies $g_\ell  {\overset{||\cdot ||_\vartheta}{\longrightarrow}} 0$, then $\lim_{\ell\to \infty} g_\ell(\xxx) =0$ for all $\xxx \in X$ uniformly.  Indeed for $\xxx \in X$, let  $(\vartheta=\xxx_0, \xxx_1,\cdots, \xxx_n= \xxx)$ be the geodesic from $\vartheta$ to $\xxx$, then it follows from the same argument as in \eqref{eq3.1'} that
$$
 |g_\ell (\xxx)- g_\ell (\vartheta)|  \leq   \sum_{k=0}^{n-1} |g_\ell(\xxx_{k+1}) - g_\ell(\xxx_k)| \leq \Big(\sum_{k=0}^{n-1} C_1\tau^{k/2}\Big) \sqrt {{\mathcal E}_X(g_\ell)} = C_2 \sqrt {{\mathcal E}_X(g_\ell)}.
$$
Also observe that $\lim_{\ell\to \infty}g_\ell(\vartheta) =0$, and hence the claim follows.

\vspace {0.1cm}
Now for $f \in {\mathcal D}_{X, 0}$, let $\{f_\ell\}_\ell \subset {\mathcal D}_X$ be such that each $f_\ell$ has finite support and $f_\ell {\overset{||\cdot ||_\vartheta}{\longrightarrow}} f$. For $\varepsilon >0$, by the claim, there exists $\ell_0$ such that $|(f-f_{\ell_0})(\xxx)| \leq \varepsilon$ for all $\xxx \in X$. For $\xi \in K$, let $(\xxx_n)_n$ be a geodesic ray that converges to $\xi$.  Then
$$
 |f(\xxx_n)| \leq |(f-f_{\ell_0}) (\xxx_n)| + |f_{\ell_0} (\xxx_n)| \leq |f_{\ell_0} (\xxx_n)|+\varepsilon, \qquad \forall \ n.
$$
This implies  $ {\rm Tr}(f)(\xi) := \lim_{n\to\infty} f(\xxx_n) = 0$, and completes the proof.
\end{proof}

\bigskip

In Theorem \ref{th3.6}, we can actually give another norm on  ${\mathcal D}_X$ so that $H: {\mathcal D}_K\rightarrow {\mathcal H}{\mathcal D}_X$ is a Banach space isomorphism. Indeed, by Corollary \ref{th3.6'} and the continuity of functions in  ${\mathcal D}_K$, we know that functions in ${\mathcal D}_X$ are bounded. Fix $w \in (0,r^\alpha)$. Let $\Vert f \Vert_{\ell^2(X,w)}^2 = \sum_{\xxx \in X}|f(\xxx)|^2 w^{|\xxx|}$, and define $\Vert \cdot \Vert_{\en_X}$ on $\dom_X$ by
\begin{equation} \label{eq3.5'}
\Vert f \Vert_{\en_X}^2 = \en_X[f] + \Vert f \Vert_{\ell^2(X,w)}^2.
\end{equation}
Then it is direct to check that $\Vert f \Vert_{\en_X}^2$ defines a complete norm on $\dom_X$.

\medskip

\begin{corollary} \label{th3.5}
With the same assumption as in  Theorem \ref{th3.6}, let $w \in (0, r^\alpha)$.  Then for all $u \in L^2(K,\nu)$,
\begin{equation} \label{eq3.6}
\Vert Hu \Vert_{\ell^2(X,w)} \leq C \Vert u \Vert_{L^2(K,\nu)}.
\end{equation}
Consequently, $H: ({\mathcal D}_K, \ \Vert \cdot \Vert_{\en_K})\rightarrow ({\mathcal H}{\mathcal D}_X , \ \Vert \cdot \Vert_{\en_X}) $ is an isomorphism.
\end{corollary}

\medskip

\begin{proof}  Let $F(\xxx, \yyy)$  denote the probability that the random walk ever visits $\yyy$ from $\xxx$.
For $n \geq 1$ and $|\yyy| > n$, by \cite[Theorem 4.6]{KLW},
\begin{equation*} \label{eq3.7}
F(\vartheta, \yyy) = \sum_{\xxx \in {\mathcal J}_n}F_n(\vartheta, \xxx)F(\xxx,\yyy) = \sum_{\xxx \in {\mathcal J}_n} r_\xxx^\alpha F(\xxx,\yyy) \geq r^{\alpha(n+1)} \sum_{\xxx \in {\mathcal J}_n} F(\xxx,\yyy).
\end{equation*}
Hence $\sum_{\xxx \in {\mathcal J}_n} K(\xxx,\xi) = \sum_{\xxx \in {\mathcal J}_n}\frac {F(\xxx, \yyy)}{F(\vartheta, \yyy)} \leq r^{-\alpha(n+1)}$.
It follows that for $u \in L^2(K,\nu)$,
\begin{align*}
\Vert Hu \Vert_{\ell^2(X,w)}^2 &\ = \ {\sum}_{\xxx \in X} \big(\mathbb{E}_\xxx(u(Z_\infty))\big)^2 w^{|\xxx|}
\ \leq \ {\sum}_{\xxx \in X} \big (\mathbb{E}_\xxx(u(Z_\infty)^2)\big )                                                                    w^{|\xxx|} \hspace{9mm} \\
&\ =\  {\sum}_{n=0}^\infty w^n \sum_{\xxx \in {\mathcal J}_n} \int_K K(\xxx,\xi)|u(\xi)|^2 d\nu(\xi)
\ \leq \ C \Vert u \Vert_{L^2(K,\nu)}^2.
\end{align*}
where $C= r^{-\alpha}{\sum}_{n=0}^\infty (w/r^\alpha)^n$. As $w/r^\alpha<1$, this yields \eqref{eq3.6}.  In view of Theorem \ref{th3.6}, the norm isomorphism of the map  $H: {\mathcal D}_K \rightarrow {\mathcal H}{\mathcal D}_X$ follows from this and $\en_K (u, v) = \en_X (Hu, Hv)$, and the open mapping theorem.
\end{proof}

\bigskip

\section{Effective resistances of ${\mathcal E}_X$}
\label{sec:4}

\noindent In this section, we will set up the limiting resistance for the $\lambda$-NRW on the augmented tree $(X,\mathfrak E)$ in order to prepare for the investigation of the critical exponents of ${\mathcal D}_K$ in the next section.

\medskip

We will start with a general situation. Let $V$ be a finite graph with a reversible Markov chain with conductance $c(x, y), x,y \in V$. Let $\ell (V)$ denote the class of real valued functions on $V$, and let ${\mathcal E}_V(f)$ be the graph energy of $f$.  For any $V_1 \subset V$, it is well-known that each $f\in \ell(V_1)$ has a harmonic extension $\widetilde f$ on $V$;  $\widetilde f$ has the minimal energy among all $g\in \ell (V)$ with $g|_{V_1}=f$, and the harmonicity for $x \in V\setminus V_1$ implies
\begin{equation} \label{eq4.0}
 {\sum}_{y\sim x} c(x, y) (\widetilde f(x) - \widetilde f(y)) =0, \qquad x \in V\setminus V_1.
\end{equation}
In the following, we give an expression of the minimal energy in terms of the conductance $c(x,y)$ of the chain on $V$.
\medskip

\begin{proposition} \label{th4.0}
Let $V$ be a finite set, and $V = V_1\cup V_2$ with $\# V_1 \geq 2$. Assume that there is a reversible Markov chain on $V$ with conductance $c(\cdot, \cdot)$. Then for $f \in V_1$,
\begin{equation} \label{eq5.03}
\min \big \{\mathcal E_V[g]: g \in \ell(V), g |_{V_1}=f \big \} = \frac 12 {\sum}_{x,y \in  V_1, x \neq y} c_\ast(x,y)(f(x)-f(y))^2,
\end{equation}
where $c_\ast(x,y) = c(x,y)+\sum_{z,w \in V_2} c(x,z)G_{ V_2}(z,w)P(w,y)$, $x,y \in  V_1$, $x \neq y$ (here $G_{ V_2}(\cdot,\cdot)$ is the Green function of the random walk restricted to $ V_2$), and it defines a conductance function on $V_1$.
\end{proposition}

\begin{proof}  Let $F^{ V_1}(x,y)=\mathbb P_x(Z_{t_{ V_1}}=y)$ where $t_{ V_1}$ is the first hitting time of $ V_1$. Then
$$
F^{V_1}(z,y)= {\sum}_{w\in V_2}G_{ V_2}(z,w)P(w,y),  \quad \forall \, z\in V_2, \ y \in V_1.
$$
 We can check directly from the definition that  $c_*(x, y) = c_*(y, x), \ x, y \in V_1,$
 using  the reversibility of the chain (i.e., $m(x)P(x,z) = m(z)P(z, x)$ and $m(z) G_{V_2} (z, w) = m(w)  G_{V_2} (w, z)$).
Hence $c_*(x, y)$ defines a conductance on $V_1$.

\vspace {0.1cm}

To prove \eqref{eq5.03}, we let $h(\cdot) = \sum_{y \in  V_1} F^{ V_1}(\cdot,y)f(y) \in \ell(V)$. Then it is easy to check that $h$ is the unique function such that $Ph=h$ on $V_2$ and $h=f$ on $ V_1$. Hence $\mathcal E_V[h] = \min \{\mathcal E_V[g]: g \in \ell(V), g |_{ V_1}=f\}$. Observe that
$$
\mathcal E_V[h] = \frac 12 \sum_{x,y \in V}c(x,y)(h(x)-h(y))^2 = \sum_{x,y \in V} c(x,y)(h(x)-h(y))h(x)
$$
Hence
\begin{align*}
&\mathcal E_V[h] = \sum_{x \in  V_1} h(x) \sum_{y \in V}c(x,y)(h(x)-h(y)) \quad \  \hbox{(by $Ph=h$ on $V_2$)} \\
&= \sum_{x \in  V_1} f(x) \Big(\sum_{y \in V_1}c(x,y)(f(x)-f(y)) + \sum_{y \in V_2} c(x,y)\sum_{z \in V_1} F^{V_1}(y,z)(f(x)-f(z))\Big)  \\
&= \sum_{x,y \in  V_1} f(x)(f(x)-f(y))\Big(c(x,y)+\sum_{z \in V_2}c(x,z) F^{V_1}(z,y)\Big)  \quad (\hbox {switch  $y$ and $z$})\\
&= \sum_{x,y \in  V_1} c_\ast(x,y)f(x)(f(x)-f(y)) \\
&= \frac 12 \sum_{x,y \in  V_1} c_\ast(x,y)(f(x)-f(y))^2. \qquad \  \hbox{(use $c_\ast(x,y)=c_\ast(y,x)$)}
\end{align*}
This yields \eqref{eq5.03}.
\end{proof}

For a finite connected graph $(X,\mathfrak E)$ with conductances, the {\it effective resistance} between two disjoint nonempty subsets $E,F \subset X$ is given by
\begin{equation} \label{eq4.1}
R_X(E,F) = (\min\{\en_X[f]:f \in \ell(X) \hbox{ with } f=1 \hbox{ on } E, \hbox{ and } f=0 \hbox{ on } F\})^{-1}.
\end{equation}
Also we set $R_X(E,F)=0$ if $E \cap F \neq \emptyset$ by convention. Clearly $R_X(\cdot, \cdot)$ is symmetric; the energy minimizer in \eqref{eq4.1} is unique,  bounded in between $0$ and $1$,  and is harmonic on $X\setminus(E\cup F)$.

\bigskip

For the $\lambda$-NRW on $(X, {\mathfrak E})$, for convenience and the simplicity in the estimations, we will assume slightly more that the conductance  on the horizontal edges satisfies
\begin{equation} \label{eq4.2}
c(\xxx,\yyy) = r^{\alpha|\xxx|}\lambda^{-|\xxx|}\quad \hbox{for } \xxx \sim_h \yyy \in X \setminus \{\vartheta\}.
\end{equation}
(we use $\asymp$ in \eqref{eq2.1}), and there is no change of the results. Let $\{\kappa_n\}_{n=0}^\infty$ be a  {\it $\kappa$-sequence} defined in Section \ref{sec:2}, i.e., each $\kappa_n$ is a selection map from $K$ to ${\mathcal J}_n$ such that for each $\xi \in K$, $\{\kappa_n(\xi)\}_{n=0}^\infty$ is a geodesic ray that converges to $\xi$. For any two closed subsets $\Phi$, $\Psi \subset K$, we define the {\rm level-$n$ resistance} between them (depend on $\kappa_n$)  by
\begin{equation} \label{eq4.3}
R_n^{(\lambda)}(\Phi, \Psi) := R_{X_n}(\kappa_n(\Phi), \kappa_n(\Psi)),
\end{equation}
where $X_n := \bigcup_{k=0}^n \mathcal J_k$ and has same conductance restricted from $X$.

\medskip

\begin{theorem} \label{th4.1}
Suppose $\{S_i\}_{i=1}^N$ satisfies the OSC, and let $\{Z_n\}_{n=0}^\infty$ be a $\lambda$-NRW on the augmented tree $(X,\mathfrak E)$ with $\lambda \in (0,r^\alpha)$. Then for any two closed subsets $\Phi$, $\Psi \subset K$, the limit $\lim_{n \to \infty} R_n^{(\lambda)}(\Phi, \Psi)$ exists, and is independent of the choice of the  $\kappa$-sequence.
\end{theorem}

We will prove a technical lemma first. For $E, F \subset {\mathcal J}_n$ such that in the graph distance, ${\rm dist} (E, F)>2$, we define
\begin{equation}\label{eq4.3'}
\partial E = \{\xxx \in {\mathcal J}_n: {\rm dist}(\xxx, E) =1\}, \quad \partial F = \{\xxx \in {\mathcal J}_n: {\rm dist}(\xxx, F) =1\}.
\end{equation}
Let ${\mathcal E}_n := {\mathcal E}_n(E, F) = \min\{\en_{X_n}[f]:f \in \ell(X_n), \   f=1 \hbox{ on } E, \ f=0 \hbox{ on } F\}$, and let $f_n$ be the corresponding energy minimizing function.

\medskip

\begin{lemma} \label{lem} Consider the $\lambda$-NRW on $(X, {\mathfrak E})$ with $\lambda \in (0, r^\alpha)$. Let $\{E_n\}_{n \geq 1}, \{F_n\}_{n \geq 1}$ be two sequences such that $E_n,F_n \subset \mathcal J_n$, and $\liminf_{n \to \infty} {\rm dist} (E_n, F_n)>2$. If $\sup_{n \geq 1} \en_n (E_n,F_n)< \infty$, then for any $\varepsilon >0$, there exists $n_0$ such that for $n\geq n_0$,
$$
\sum_{\xxx \in \partial E_n}\sum_{\yyy \in X_n \setminus E_n}c(\xxx, \yyy) (1-f_n(\yyy))^2 < \varepsilon, \ \ \sum_{\xxx \in \partial F_n}\sum_{\yyy \in X_n \setminus F_n}c(\xxx, \yyy) f_n(\yyy)^2 < \varepsilon,
$$
where $f_n$ is the energy minimizer of ${\mathcal E}_n :={\mathcal E}_n (E_n, F_n)$.
\end{lemma}

\begin{proof} We write $f=f_n$ for simplicity. We observe that
\begin{align} \label{eq4.7'}
\en_n &= \frac 12 \sum_{\xxx,\yyy \in X_n} c(\xxx,\yyy)(f(\xxx)-f(\yyy))^2  \nonumber\\
&= \sum_{\xxx,\yyy \in X_n} c(\xxx,\yyy)f(\xxx)(f(\xxx)-f(\yyy)) \nonumber\\
&= \sum_{\xxx\in E_n \cup F_n} f(\xxx)\sum_{ \yyy \in X_n} c(\xxx,\yyy)(f(\xxx)-f(\yyy))  \qquad (\hbox {by} \ \eqref {eq4.0}) \nonumber\\
&= \sum_{\xxx \in E_n} \sum_{\yyy \in X_n} c(\xxx,\yyy)(1-f(\yyy))
\ \geq \ (r^\alpha / \lambda)^n \sum_{\yyy \in \partial E_n}(1-f(\yyy)).
\end{align}
Thus $f(\xxx) \geq 1-(\lambda/r^\alpha)^n \en_n$ for $\xxx \in \partial E_n$. Using a similar argument, and that $f$ is harmonic on $X_n \setminus (E_n  \cup F_n \cup \partial E_n)$,  for large n, we have
\begin{align} \label{eq4.8'}
\en_n &\geq \frac 12 \sum_{\xxx,\yyy \in X_n} c(\xxx,\yyy)(f(\xxx)-f(\yyy))^2 - \sum_{\xxx \in \partial E_n} \sum_{\yyy \in E_n} c(\xxx,\yyy)(f(\xxx)-f(\yyy))^2 \nonumber \\
&=\Big(\sum_{\xxx \in {E_n \cup F_n\cup \partial E_n}} \sum_{\yyy \in X_n} - \sum_{\xxx \in {E_n }} \sum_{\yyy \in \partial E_n} -\sum_{\xxx \in \partial E_n} \sum_{\yyy \in E_n}\Big) c(\xxx,\yyy)f(\xxx)(f(\xxx)-f(\yyy)) \nonumber \\
&= \sum_{\xxx \in \partial E_n}f(\xxx) \sum_{\yyy \in X_n \setminus E_n} c(\xxx,\yyy)(f(\xxx)-f(\yyy)) + \sum_{\yyy \in E_n} c(\yyy,\yyy^-)(1-f(\yyy^-)) \nonumber \\
&\geq \big(1-(\lambda/r^\alpha)^n \en_n \big) \sum_{\xxx \in \partial E_n}\sum_{\yyy \in X_n \setminus E_n} c(\xxx,\yyy)(f(\xxx)-f(\yyy)).
\end{align}
(The last inequality holds because for  $\xxx \in \partial E_n$, $\sum_{\yyy \in X_n \setminus E_n} c(\xxx,\yyy)(f(\xxx)-f(\yyy)) \geq 0$, as by harmonicity,
$
\sum_{\yyy \in X_n \setminus E_n} \cdots = - \sum_{\yyy \in E_n} \cdots  = - \sum_{\yyy \in E_n} c(\xxx,\yyy)(f(\xxx)-1)\geq 0
$.)

\vspace {0.1cm}

Now we use \eqref{eq4.7'} and \eqref{eq4.8'} to make the final estimate:
\begin{align} \label{eq4.9'}
&\hspace{4.5mm} \sum_{\xxx \in \partial E_n} \sum_{\yyy \in X_n \setminus E_n} c(\xxx,\yyy)(1-f(\yyy))^2 \leq \Big (\sum_{\xxx \in \partial E_n} \sum_{\yyy \in X_n \setminus E_n} c(\xxx,\yyy)^{1/2}(1-f(\yyy))\Big)^2 \nonumber \\
&\leq \frac{\lambda^n}{r^{\alpha(n+1)}}\Big(\sum_{\xxx \in \partial E_n} \sum_{\yyy \in X_n \setminus E_n} c(\xxx,\yyy)(1-f(\yyy))\Big)^2
\nonumber \\
&= \frac{\lambda^n}{r^{\alpha(n+1)}} \Big(\sum_{\xxx \in \partial E_n} \sum_{\yyy \in X_n \setminus E_n} c(\xxx,\yyy)\big((1-f(\xxx))+(f(\xxx)-f(\yyy))\big)\Big)^2 \nonumber\\
&\leq \frac{\lambda^n}{r^{\alpha(n+1)}}\left(k\en_n + \frac{\en_n}{1-(\lambda/r^\alpha)^n \en_n}\right)^2 =: \varepsilon(n) \hspace{8mm} \hbox{\big(by \eqref{eq4.7'}, \eqref{eq4.8'}\big)}
\end{align}
 where $k=\sup_{\xxx \in X} \# \{\yyy : \xxx \sim_h \yyy \}$ (as the graph $(X,\ed)$ has bounded degree, and $c(\xxx, \yyy)>0$ only when $\xxx \sim_h \yyy$ or $\yyy = \xxx^-$). Hence we can choose $n_0$ such that $\varepsilon (n) <\varepsilon$  for $n>n_0$.
 Analogously, using $1-f$ instead of $f$, we obtain the estimate for $F$ as well.
\end{proof}

\medskip

\begin{proof}[Proof of Theorem \ref{th4.1}]
We fix a $\lambda \in (0, r^\alpha)$ and omit the superscript $(\lambda)$ in this proof. First we fix a $\kappa$-sequence $\{\kappa_n\}_{n=0}^\infty$, and prove that  $\lim_{n \to \infty} R_n(\Phi, \Psi)$ exists.  For brevity, we write $\Phi_n :=\kappa_n(\Phi)$ and $\Psi_n := \kappa_n(\Psi)$. If $\Phi \cap \Psi \neq \emptyset$, then by the property of geodesic rays in $(X, {\mathfrak E})$, for any $n$, either $\Phi_n \cap \Psi_n \neq \emptyset$ or $\min\{d(\xxx,\yyy): \xxx \in \Phi_n,\, \yyy \in \Psi_n\} = 1$ (by Lemma \ref{th2.0}). In both situations, we have $\lim_{n \to \infty} R_n(\Phi, \Psi)=0$ (for the second case,  by \eqref{eq4.1}, $R_n(\Phi, \Psi) \leq (r^{\alpha n} \lambda^{-n})^{-1} = (\lambda / r^\alpha)^n$).

\vspace{1mm}

 Hence we assume that  $\Phi \cap \Psi = \emptyset$. Then there exists $\ell>0$ such that for $n\geq \ell$,  ${\rm dist}(\Phi_n, \Psi_n) > 3$.
By \eqref{eq4.1} and \eqref{eq4.3}, for $n \geq \ell$,
\begin{equation} \label{eq4.7}
R_n(\Phi,\Psi) = (\min \{\en_{X_n}[f]: f=1 \hbox{ on } \Phi_n, \hbox{ and } f=0 \hbox{ on } \Psi_n\})^{-1}.
\end{equation}
Let $\en_n$ denote the minimal energy, and let $f_n \in \ell(X_n)$ be the energy minimizer in \eqref{eq4.7}.
Let $\{n_k\}_{k \geq 1}$  with $n_k \geq \ell$ be the subsequence such that
$
\lim_{k \to \infty} R_{n_k}(\Phi, \Psi) = \\ \limsup_{n \to \infty} R_n(\Phi, \Psi)>0.
$
(otherwise $\lim_{n \to \infty} R_n(\Phi, \Psi)=0$). Then $\sup_k \en_{n_k} < \infty$. For $n < n_k$ and $\xi \in K$, by Lemma \ref{th3.3}(i), we have
\begin{align} \label{eq4.5'}
|f_{n_k}(\kappa_n(\xi)) - f_{n_k}(\kappa_{n_k}(\xi))| & \ \leq \ \sum_{m=1}^{n_k-n}  |f_{n_k}(\kappa_{n+m-1}(\xi))-f_{n_k}(\kappa_{n+m}(\xi))| \nonumber\\
& \leq \ C\Big (\frac {\lambda}{r^\alpha}\Big)^{n/2}  := \varepsilon(n).
\end{align}
As $\lambda \in (0,r^\alpha)$, $\lim_{n \to \infty} \varepsilon(n) = 0$.
Let $V_1 = \Phi_n \cup \Psi_n$. Then for sufficiently large $n$ and $n_k > n$, we have $\varepsilon(n) < \frac 12$, and
\begin{align*}
\en_{n_k} &\geq \en_{X_n}[f_{n_k}] \geq \min\{\en_{X_n}[f]: f \in \ell(X_n), f=f_{n_k} \hbox{ on } V_1\} \\
&= \frac 12 \sum_{\xxx,\yyy \in V_1} c_*(\xxx,\yyy)(f_{n_k}(\xxx)-f_{n_k}(\yyy))^2 \hspace{17mm} \hbox{(by Proposition \ref{th4.0})} \\
&\geq \sum_{\xxx \in \Phi_n} \sum_{\yyy \in \Psi_n} c_*(\xxx,\yyy)(f_{n_k}(\xxx)-f_{n_k}(\yyy))^2 \\
&\geq \sum_{\xxx \in \Phi_n} \sum_{\yyy \in \Psi_n} c_*(\xxx,\yyy)((1-2\varepsilon(n))^2
\hspace{23mm} \hbox{(by \eqref{eq4.5'})}\\
&= \min\{\en_{X_n}[f]: f \in \ell(X_n), f|_{\Phi_n} = 1-2\varepsilon(n), f|_{\Psi_n} = 0\} \\
&\hspace{80mm} \hbox{(by Proposition \ref{th4.0})} \\
&= (1-2\varepsilon(n))^2 \en_n.
\end{align*}
Therefore,  $R_n(\Phi, \Psi) \geq (1-2\varepsilon(n))^2 R_{n_k}(\Phi, \Psi)$ for any large $n$ and $n_k > n$. Taking limit, we have
$$\liminf_{n \to \infty} R_n(\Phi, \Psi) \geq \lim_{k \to \infty} R_{n_k}(\Phi, \Psi) = \limsup_{n \to \infty} R_n(\Phi, \Psi).$$
Hence $\lim_{n \to \infty} R_n(\Phi, \Psi)$ exists.

\vspace{2mm}

Next we show that the above limit is independent of the choice of the $\kappa$-sequence.
For this, we define
\begin{equation*}
\partial\Phi_n = \{\xxx \in {\mathcal J}_n: d(\xxx,\Phi_n)=1\}, \quad
\partial\Psi_n = \{\xxx \in {\mathcal J}_n: d(\xxx,\Psi_n)=1\}
\end{equation*}
as in \eqref{eq4.3'}. For any other $\kappa$-sequences $\{\kappa_n'\}_n$, it follows from Lemma \ref{th2.0} that $\kappa'_n(\Phi) \subset \Phi_n \cup \partial \Phi_n$ and $\kappa'_n(\Psi) \subset  \Psi_n \cup \partial \Phi_n$. Hence it suffices to show that
\begin{equation} \label{4.7}
\lim_{n \to \infty} R_{X_n}(\Phi_n \cup \partial\Phi_n, \Psi_n \cup \partial \Phi_n) = \lim_{n \to \infty} R_n(\Phi,\Psi).
\end{equation}
Without loss of generality, we assume $\lim_{n \to \infty} R_n(\Phi,\Psi)>0$. Then $\sup_n \en_n < \infty$.
Let $h_n \in \ell(X_n)$ with $h_n=1$ on $\partial\Phi_n$, $h_n=0$ on $\partial\Psi_n$, and $h_n=f_n$ on $X_n \setminus (\partial\Phi_n \cup \partial\Psi_n)$.
\begin{equation*}
0 \leq R_{X_n}(\Phi_n \cup \partial\Phi_n, \Psi_n \cup \partial \Phi_n)^{-1} - R_n(\Phi,\Psi)^{-1} \leq \en_{X_n}[h_n]-\en_{X_n}[f_n] .
\end{equation*}
Then by Lemma \ref{lem}, for given $\varepsilon$, and for large $n$,  $\en_{X_n}[h_n]-\en_{X_n}[f_n] \leq 2\varepsilon$. This implies
\eqref{4.7} and proves the theorem.
\end{proof}

\bigskip
Theorem \ref {th4.1} implies the following definition is well defined.

\begin{definition} \label{de4.2}
With the same assumption as in Theorem \ref{th4.1}, we define the {\rm limiting resistance} between two closed subsets $\Phi$ and $\Psi$ in $K$ by
\begin{equation} \label{eq4.4}
R^{(\lambda)}(\Phi, \Psi) := \lim_{n \to \infty} R_n^{(\lambda)}(\Phi, \Psi).
\end{equation}
(We omit the superscript $(\lambda)$ if there is no confusion.)
\end{definition}

\section{The critical exponents of  ${\mathcal D}_X$}
\label{sec:5}

\noindent We first establish a basic result on the existence of nonconstant functions in $\dom_K$.

\medskip

\begin{theorem} \label{th4.2}
With the same assumption as in Theorem \ref{th4.1}, suppose $\Phi,\Psi $ are two closed subsets of $K$ satisfying $R(\Phi,\Psi)>0$. Then there exists $u :=u_{\Phi,\Psi}  \in \dom_K$ such that $u=1$ on  $\Phi$, and  $u=0$ on $\Psi$. Moreover, $u_{\Phi,\Psi}$ is the unique energy minimizer in ${\mathcal D}_K$ in the following sense
\begin{align} \label{eq4.15'}
R(\Phi,\Psi)^{-1} = \en_K[u_{\Phi,\Psi}] = \inf\{\en_K[u']: u' \in \dom_K \hbox{ with } u'=1 \hbox{ on } \Phi, \ u'= 0  \hbox { on }  \Psi \}.
\end{align}
\end{theorem}

\begin{proof} First we show that the set on the right hand side in \eqref{eq4.15'} is non-empty. Clearly $\Phi \cap \Psi = \emptyset$ (otherwise $R(\Phi,\Psi) = 0$). Fix a $\kappa$-sequence $\{\kappa_n\}_n$. As in the proof of Theorem \ref{th4.1}, there exists a positive integer $\ell$ such that $\kappa_n(\Phi) \cap \kappa_n(\Psi) = \emptyset$ for all $n \geq \ell$,  let $f_n \in \ell(X_n)$ be the energy minimizer for $\kappa_n(\Phi)$ and $\kappa_n(\Psi)$ as in \eqref{eq4.7}. We extend $f_n$ to $X$ by setting $f_n(\xxx) = 0$ for $\xxx \in X \setminus X_n$, then $f_n$ is harmonic on $X_{n-1}$. Note that $0 \leq f_n \leq 1$ for all $n \geq \ell$. Hence for each $\xxx \in X$, there exists a convergent subsequence of $\{f_n(\xxx)\}_{n \geq \ell}$. By the diagonal argument, we can find a subsequence $\{f_{n_k}\}_{k \geq 1}$ with $n_1 \geq \ell$ such that $f_{n_k}$ converges to a function $f=: f_{\Phi, \Psi} \in \ell(X)$ pointwise. We claim that
\begin{enumerate}
\item[(a)] $f \in {\mathcal H}{\mathcal D}_X$ and $0 \leq f \leq 1$ on $X$;
\item[(b)] For any $\xi \in \Phi$, $\lim_{n \to \infty} f(\kappa_n(\xi)) = 1$;
\item[(c)] For any $\eta \in \Psi$, $\lim_{n \to \infty} f(\kappa_n(\eta)) = 0$.
\end{enumerate}
In fact, as $f_{n_k}$ is harmonic on $X_{n_k-1}$, the pointwise limit $f$ is harmonic on $X$. For $k \geq 1$, let $g_k$ be the function on the edge set $\ed$ defined by: for $(\xxx,\yyy) \in \ed$,
\begin{equation*}
g_k(\xxx,\yyy) = \begin{cases}
\ c(\xxx,\yyy)(f_{n_k}(\xxx)-f_{n_k}(\yyy))^2, & \quad \hbox{if } \xxx,\yyy \in X_{n_k}, \\
\ 0, & \quad \hbox{otherwise}.
\end{cases}
\end{equation*}
Then $\en_{n_k}:=\en_{X_{n_k}}[f_{n_k}] = \frac{1}{2}\sum_{(\xxx,\yyy) \in \ed} g_k(\xxx,\yyy)$, and $\lim_{k \to \infty} g_k(\xxx,\yyy) = c(\xxx,\yyy)(f(\xxx)-f(\yyy))^2$. By Fatou's Lemma, we have
\begin{align} \label{eq4.14'}
\en_X[f] &= \dfrac{1}{2}\sum_{(\xxx,\yyy) \in \ed} c(\xxx,\yyy)(f(\xxx)-f(\yyy))^2 = \dfrac{1}{2}\sum_{(\xxx,\yyy) \in \ed} \left(\lim\limits_{k \to \infty} g_k(\xxx,\yyy)\right) \nonumber \\
&\leq \dfrac{1}{2}\liminf_{k \to \infty} \sum_{(\xxx,\yyy) \in \ed} g_k(\xxx,\yyy) = \lim_{k \to \infty} \en_{n_k} = R(\Phi,\Psi)^{-1} < \infty.
\end{align}
Hence (a) follows.

\vspace{0.1cm}

  To prove (b), observe that $R(\Phi,\Psi) > 0$ implies that $\sup_{k \geq 1} \en_{n_k}< \infty$. Hence for any $k \geq 1$, $n < n_k$ and $\xi \in \Phi$, by Lemma \ref{th3.3}(i)
\begin{align*}
|f_{n_k}(\kappa_n(\xi))-1| &\leq \sum_{m=n}^{n_k-1}|f_{n_k}(\kappa_m(\xi))-f_{n_k}(\kappa_{m+1}(\xi))| \leq C_1(\lambda/r^\alpha)^{n/2}.
\end{align*}
Letting $k \to \infty$, we have $|f(\kappa_n(\xi))-1|\leq C_2(\lambda/r^\alpha)^{n/2}$, hence (b) follows by letting $n \to \infty$. With a similar argument, we can also conclude (c).

\vspace{0.3cm}

 By the claim and Theorem \ref{th3.6}, let $u = {\rm Tr}f \in \dom_K$. Then $0 \leq u \leq 1$ on $K$, $u(\xi)=\lim_{n \to \infty} f(\kappa_n(\xi))=1$ for all $\xi \in \Phi$, and $u(\eta)=\lim_{n \to \infty} f(\kappa_n(\eta))=0$ for all $\eta \in \Psi$.

 \medskip

 Now we complete the proof of the theorem. By \eqref{eq4.14'}, $\en_K[u_{\Phi,\Psi}] = \en_X[f_{\Phi,\Psi}] \leq R(\Phi,\Psi)^{-1}$. For the reverse inequality, it suffices to show that $R(\Phi,\Psi)^{-1} \leq \en_K[u]$ for all $u \in \dom_K$ with $u=1$ on $\Phi$ and $u=0$ on $\Psi$. Fix a $\kappa$-sequence $\{\kappa_n\}_n$. For any $\varepsilon \in (0,\frac 12)$, by Proposition \ref{th3.2}, there exists a positive integer $n_0 = n_0(\varepsilon)$ such that $|Hu(\kappa_n(\xi))-u(\xi)| \leq \varepsilon$ whenever $n \geq n_0$ and $\xi \in K$. Taking $V_1 = \kappa_n(\Phi) \cup \kappa_n(\Psi)$ and $g = Hu|_{V_1}$ as in Proposition \ref{th4.0}, then we have, for $n \geq n_0$,
\begin{equation} \label{eq4.16'}
\min\{\en_{X_n}[f]: f \in \ell(X_n), f=Hu \hbox{ on } V_1\} = \frac 12 \sum_{\xxx,\yyy \in V_1} c_*(\xxx,\yyy)(Hu(\xxx)-Hu(\yyy))^2.
\end{equation}
Hence
\begin{align*}
\en_{X_n}[Hu] &\geq \min\{\en_{X_n}[f]: f \in \ell(X_n), f=Hu \hbox{ on } V_1\} \\
&\geq \sum_{\xxx \in \kappa_n(\Phi)} \sum_{\yyy \in \kappa_n(\Psi)}c_*(\xxx,\yyy)(Hu(\xxx)-Hu(\yyy))^2 \hspace{8mm} \hbox{(by \eqref{eq4.16'})} \\
&\geq \sum_{\xxx \in \kappa_n(\Phi)} \sum_{\yyy \in \kappa_n(\Psi)}c_*(\xxx,\yyy)\big((1-\varepsilon)-\varepsilon\big)^2 \\
&= (1-2\varepsilon)^2 R_n(\Phi,\Psi)^{-1}.
\end{align*}
As $\varepsilon$ can be arbitrarily small, we have $ R(\Phi,\Psi)^{-1}\leq \lim_{n \to \infty} \en_{X_n}[Hu] = \en_K[u]$. Hence \eqref{eq4.15'} follows.

The uniqueness of $ u_{\Phi,\Psi}$ as an energy minimizer follows from the fact that  ${\mathcal E}_K$ is strictly convex in ${\mathcal D}_K$.
\end{proof}

\bigskip

The function $f_{\Phi, \Psi} \in {\mathcal H}{\mathcal D}_X$ thus constructed is called a {\it harmonic function induced by $\Phi$ and $\Psi$}. The function $u_{\Phi, \Psi} = {\rm Tr}f_{\Phi, \Psi} \in \dom_K$ is referred as the {\it energy minimizer of $\Phi$ and $\Psi$}.

\medskip

\begin{corollary} \label{th4.3'}
With the same assumption as in Theorem \ref{th4.1}, the following conditions are equivalent: for two distinct points $\xi,\eta \in K$,
\begin{enumerate}
\item[(i)] there exists $u \in \mathcal D_K$ with range $[0,1]$ such that $u(\xi)=1$ and $u(\eta)=0$;

\item[(ii)] there exists $u \in \mathcal D_K$ such that $u(\xi) \neq u(\eta)$;

\item[(iii)] $R(\xi, \eta)>0$.
\end{enumerate}
In this case,
$
R(\xi,\eta) = \sup\big\{\dfrac{|u(\xi)-u(\eta)|^2}{\en_K(u,u)}: \ u \in \dom_K, \ \en_K(u,u)>0\big\}.
$
\end{corollary}

\begin {proof}
Note that (i) $\Rightarrow$ (ii) is trivial, and (iii) $\Rightarrow$ (i) follows from Theorem \ref{th4.2}. We need only prove (ii) $\Rightarrow$ (iii). We observe that  the given $u\in {\mathcal D}_K$ is continuous (Proposition \ref{th2.4}). Fix any $\kappa$-sequence, by Corollary \ref{th3.2}, there exists  $n_0>0$ such that for $n \geq n_0$,
$
|Hu(\kappa_n(\xi))-u(\xi)|\leq \frac13 {|u(\xi)-u(\eta)|}$,  and the same for $\eta$.
Hence $|Hu(\kappa_n(\xi))-Hu(\kappa_n(\eta))|\geq \frac{1}{3}|u(\xi)-u(\eta)|$. Then by \eqref{eq4.1}, for $n> n_0$
\begin {align*}
& \frac{|u(\xi)-u(\eta)|^2}{9R_n(\xi, \eta)}  \leq \frac{|Hu(\kappa_n(\xi))-Hu(\kappa_n(\eta))|^2}{R_n(\xi, \eta)} \\
\leq &\  c(\kappa_n(\xi), \kappa_n(\eta))\ |Hu(\kappa_n(\xi))-Hu(\kappa_n(\eta))|^2 \leq \en_{X_n}[Hu] .
\end{align*}
Taking the limit on $n$, we have $ \frac{|u(\xi)-u(\eta)|^2}{9{R}(\xi,\eta)} \leq \en_K[u] <\infty$. Hence $R(\xi,\eta)>0$.
\end{proof}

\medskip
\begin{corollary} \label{th4.3}
With the same assumption as in Theorem \ref{th4.1},
if $ R^{(\lambda)}(\xi,\eta)=0$ for some $\xi, \eta \in K$, then $\beta_1^* \leq \log \lambda / \log r$ where $\beta_1^*:=  \sup \{\beta > 0: {\mathcal D}^{(\beta)}_K \cap C(K) \hbox{ is dense in } C(K)\}$.
\end{corollary}

\begin{proof}
 If $ R^{(\lambda)}(\xi,\eta)=0$,  then every $u \in \dom_K$ must satisfy $u(\xi) = u(\eta)$, so  $\dom_K $ is not dense in $C(K)$, which implies $\beta_1^* \leq \log \lambda / \log r$.
\end{proof}

\medskip

\noindent {\bf Remark}. For the implication of (ii) $\Rightarrow$ (iii) in Corollary \ref{th4.3'}, we can omit $\lambda \in (0, r^\alpha) $ (i.e., $\beta > \alpha$), instead consider $u\in {\mathcal D}_K \cap C(K)$, and replace $R(\xi, \eta)$ by $\underline R(\xi, \eta): = \liminf_{n\to \infty}R_n(\xi, \eta)$, then the implication still holds. Consequently, Corollary \ref{th4.3} is still valid.

\medskip

In the following, we will apply Corollary \ref{th4.3'} to give some criteria to determine the critical exponents for  $\beta_2^*:= \sup \{\beta>0: \dim {\mathcal D}^{(\beta)}_K = \infty\}$ and
$\beta_3^*:=  \sup \{\beta > 0: {\mathcal D}^{(\beta)}_K \hbox{ contains nonconstant functions}\}$.

\medskip
Let $\iii_n = ii\cdots i \in \mathcal J_n$ denote the unique word in level $n$  consisting of symbol $i \in \Sigma$, and let $i^\infty=ii\cdots \in \Sigma^\infty$ (identified with the unique point in $\bigcap_{n \geq 0} S_{{\bf i}_n} (K)$).
Then for two distinct symbols $i$, $j \in \Sigma$, we use $R(i^\infty,j^\infty) $ to denote the limiting resistance for the corresponding two points in $K$, and  $R(i^\infty,j^\infty)=\lim_{n \to \infty}R_n(i^\infty,j^\infty)$.

\medskip

\begin{theorem} \label{th4.4}
With the same assumption as in Theorem \ref{th4.1}, $\dom_K$ consists of only constant functions if and only if
\begin{equation} \label{eq4.8}
R^{(\lambda)}(i^\infty, j^\infty)=0, \qquad \forall \, i, j \in \Sigma.
\end{equation}
Consequently, $\beta_3^* = \log \lambda_3^* / \log r$ if
\begin{equation} \label{eq4.9}
\lambda_3^*:= \sup \{\lambda>0: R^{(\lambda)}(i^\infty,j^\infty)=0, \ \forall\, i,j \in \Sigma\} \in (0, r^\alpha),
\end{equation}
and  $\beta_3^*=\infty$ if the above set of $\lambda$ is empty.
\end{theorem}

\begin{proof}
If for some $i,j \in \Sigma$, $R(i^\infty,j^\infty)>0$, then there exists $u \in \mathcal D_K$ with $u(i^\infty) \neq u(j^\infty)$ by Proposition \ref{th4.2} (or by Corollary \ref{th4.3'} (iii) $\Rightarrow$ (ii)). Thus it suffices to show that \eqref{eq4.8} implies $\dom_K=\{\hbox{constant functions}\}$.

\vspace{1mm}

First we claim that for $u \in C(K)$,  if $u(\xxx i^\infty) = u(\xxx j^\infty)$ for any $\xxx \in \Sigma^\ast$ and $i,j \in \Sigma$, then $u$ is a constant function. Indeed, let $c=u(1^\infty)$, then for $\xxx = \vartheta$, by assumption we have $u(i^\infty)=c$ for any $i \in \Sigma$. Next for  $\xxx=i$, by assumption again,  we have $u(ij^\infty)=u(i^\infty)=c$ hence $u(\xxx j^\infty)=c$ for any $\xxx \in \Sigma^1$ and $j \in \Sigma$. Following the same argument inductively, we have $u(\xxx j^\infty)=c$ for any $\xxx \in \Sigma^\ast$ and $j \in \Sigma$. By continuity, $u \equiv c$, a constant function.

\vspace{1mm}

For nonconstant $u \in C(K)$, by the claim we can pick $\xxx \in \Sigma^\ast$ and $i, j \in \Sigma$ such that $u(\xxx i^\infty) \neq u(\xxx j^\infty)$. We telescope $u$ on the cell $S_\xxx(K)$ to get $\widetilde{u} = u \circ S_\xxx$. Then $\widetilde{u}(i^\infty) \neq \widetilde{u}(j^\infty)$. By Proposition \ref{th4.3} (or by Corollary \ref{th4.3'} (ii) $\Rightarrow$ (iii)) and assumption \eqref{eq4.8}, we must have $\widetilde{u} \notin \dom_K$. Note that
\begin{align} \label{eq4.10}
\en_K[u] &\geq c_1 \ds\int_{S_\xxx(K)} \ds\int_{S_\xxx(K)} \dfrac{|u(\xi)-u(\eta)|^2}{|\xi-\eta|^{\alpha+\beta}} d\nu(\xi)d\nu(\eta) \nonumber \\
&\geq c_2 \ds\int_K \ds\int_K \dfrac{|\widetilde{u}(\xi)-\widetilde{u}(\eta)|^2}{|\xi-\eta|^{\alpha+\beta}} d\nu(\xi)d\nu(\eta) \geq c_3 \en_K[\widetilde{u}],
\end{align}
hence $u \notin \dom_K$. Finally as $\dom_K \cap C(K)=\dom_K$ by Theorem \ref{th2.1}, $\dom_K$ contains constant functions only.
\end{proof}

\bigskip

Next we will show that  $\beta_2^*=\beta_3^*$  under the connectedness of the self-similar set. The following lemma is a key step to include more non-trivial functions in ${\mathcal D}_K$.

\medskip

\begin{lemma} \label{th4.5}
With the same assumption as in Theorem \ref{th4.1}, suppose $\xi \in K$ and  $\Psi $ is a closed subset in $K$ satisfying $R(\xi,\Psi)>0$. Let $u=u_{\xi,\Psi} \in \mathcal D_K$ be the limiting harmonic function. Then for $\eta \in K$ such that $0<u(\eta)<1$, we have $R(\eta, \Psi)>0$ and $R(\xi,\Psi \cup \{\eta\})>0$.
\end{lemma}

\begin{proof}
Let $f=f_{\xi,\Psi} = Hu$ and $\varepsilon = \min \{u(\eta), 1-u(\eta)\}>0$. Fix a $\kappa$-sequence $\{\kappa_n\}_n$. By Proposition \ref{th3.2}, there exists a positive integer $m_0$ such that
\begin{equation} \label{eq4.11}
|f(\kappa_n(\eta))-u(\eta)|<\varepsilon/4,\qquad \forall\, n \geq m_0.
\end{equation}
Following the same argument as in the proof of Proposition \ref{th4.2}, let $f_n \in \ell(X_n)$ be the energy minimizer in \eqref{eq4.7} with $\Phi = \{\xi\}$. By passing to subsequence, we assume, without loss of generality,  that $f_{n} \in \ell(X)$ converges to $f$ pointwise.

\vspace{2mm}

\noindent Note that for $n \geq 1$ and $ k < n$, by Lemma \ref{th3.3}(i),
\begin{align*}
|f_{n}(\kappa_k(\eta))-f_{n}(\kappa_{n}(\eta))| \leq \sum_{m=k}^{n-1}|f_{n}(\kappa_m(\eta))-f_{n}(\kappa_{m+1}(\eta))| \leq C_1(\lambda/r^\alpha)^{k/2}.
\end{align*}
Thus we can pick a positive integer $m_1 \geq m_0$ such that
\begin{equation} \label{eq4.12}
|f_{n}(\kappa_{m_1}(\eta))-f_{n}(\kappa_{n}(\eta))| < \varepsilon/4, \qquad \forall\, n\geq m_1.
\end{equation}
Since $f_{n}(\kappa_{m_1}(\eta)) \to f(\kappa_{m_1}(\eta))$ as $n \to \infty$, there exists a positive integer $n_0$ such that $n_0 \geq m_1$ and
\begin{equation} \label{eq4.13}
|f_{n}(\kappa_{m_1}(\eta)) - f(\kappa_{m_1}(\eta))| < \varepsilon/4, \qquad \forall \, n\geq n_0.
\end{equation}
Combining \eqref{eq4.11}--\eqref{eq4.13}, we have $f_{n}(\kappa_{n}(\eta)) \in (\varepsilon/4, 1-\varepsilon/4)$ for all $n \geq n_0$.
 Using \eqref{eq4.1} and \eqref{eq4.3}, for $n \geq n_0$, we have
\begin{equation}
R_{n}(\eta,\Psi) \geq \frac{f_{n}(\kappa_{n}(\eta))^2}{\mathcal E_{X_{n}}[f_{n}]} > \frac{\varepsilon^2}{16}R_{n}(\xi,\Psi).
\end{equation}
Hence $R(\eta,\Psi) > 0$ by passing limit.

\vspace{2mm}

To prove $R(\xi,\Psi \cup \{\eta\})>0$, let $g_n \in \ell(X_n)$ be the energy minimizer in \eqref{eq4.7} with $\Phi = \{\eta\}$. By passing to subsequence if necessary, we let $\gamma_{1,n} = f_{n}(\kappa_{n}(\eta))$ and $\gamma_{2,n} = g_{n}(\kappa_{n}(\xi))$. Then $\gamma_{1,n}\gamma_{2,n} \in [0,1-\varepsilon/4)$ as $\gamma_{1,n} \in (\varepsilon/4,1-\varepsilon/4)$ (by last part) and $\gamma_{2,n} \in [0,1]$. For $n \geq 1$, we can check that the function
$$
h_n := \frac{1}{1-\gamma_{1,n}\gamma_{2,n}} f_{n} - \frac{\gamma_{1,n}}{1-\gamma_{1,n}\gamma_{2,n}} g_{n} \in \ell(X_{n})
$$
satisfies $h_n(\kappa_{n}(\xi)) = 1$, and $h_n = 0$ on $\kappa_{n}(\Psi \cup \{\eta\})$. Moreover, $h_n$ is harmonic on $X_{n} \setminus \kappa_{n}(\Psi \cup \{\xi,\eta\})$, thus $\mathcal E_{X_{n}}[h_n] = (R_{n}(\xi,\Psi \cup \{\eta\}))^{-1}$ by \eqref{eq4.3}. Hence
\begin{align*}
R(\xi, \Psi \cup \{\eta\}) &= \lim_{n \to \infty} (\mathcal E_{X_{n}}[h_n])^{-1} \geq \lim_{n \to \infty} \left(\frac{2\mathcal E_{X_{n}}[f_{n}]}{(1-\gamma_{1,n}\gamma_{2,n})^2}+\frac{2\gamma_{1,n}^2\mathcal E_{X_{n}}[g_{n}]}{(1-\gamma_{1,n}\gamma_{2,n})^2}\right)^{-1} \\
&\geq \lim_{n \to \infty} \left(\frac{2}{(\varepsilon/4)^2 R_{n}(\xi,\Psi)} + \frac{2(1-\varepsilon/4)^2}{(\varepsilon/4)^2 R_{n}(\eta,\Psi)}\right)^{-1} \\
&= \left(\frac{2}{(\varepsilon/4)^2 R(\xi,\Psi)} + \frac{2(1-\varepsilon/4)^2}{(\varepsilon/4)^2 R(\eta,\Psi)}\right)^{-1} > 0 .
\end{align*}
\end{proof}

\medskip

\begin{theorem} \label{th4.6}
With the assumptions in Theorem \ref{th4.1}, assume further $K$ is connected, and there exists $\beta>\alpha$ such that ${\mathcal D}_K (= {\mathcal D}^{(\beta)}_K$) is non-trivial. Then $\beta_2^*=\beta_3^*$.
\end{theorem}

\begin{proof}
It suffices to verify that for $\lambda \in (0,r^\alpha)$, $\dim \dom_K > 1$ (i.e., $\dom_K$ contains nonconstant functions) implies that $\dim \dom_K = \infty$. We have $R(\xi, \eta)>0$ for some $\xi,\eta \in K$ by Corollary \ref{th4.3'} (ii) $\Rightarrow$ (iii). The energy minimizer $u_1=u_{\xi,\eta} \in \dom_K$ is continuous with $u_1(\xi)=1$ and $u_1(\eta)=0$, hence there exists $\eta_1 \in K$ such that $u_1(\eta_1)=1/2$. By Lemma \ref{th4.5}, we have $R(\xi,\{\eta,\eta_1\})>0$ and this induces another energy minimizer $u_2=u_{\xi, \{\eta, \eta_1\}} \in \dom_K$ with $u_2(\xi)=1$ and $u_2(\eta)=u_2(\eta_1)=0$. By the continuity, we can pick $\eta_2 \in K$ such that $u_2(\eta_2) = 1/2$. Setting $\eta_0=\eta$ and repeating the above argument, we get a sequence of energy minimizers $\{u_n\}_{n=1}^\infty$ together with a sequence of points $\{\eta_k\}_{k=0}^\infty$ in $K$ such that $u_n(\xi)=1$, $u_n(\eta_n)=1/2$, and $u_n(\eta_k)=0$ for any $0 \leq k < n$. Thus $[u_i(\eta_j)]_{i,j \geq 1}$ is an infinite upper triangular matrix with constant diagonal entries $1/2$. Hence $\{u_n\}_{n=1}^\infty$ is a sequence of linearly independent functions in $\dom_K$, so that $\dim \mathcal D_K = \infty$.
\end{proof}

\medskip

\noindent {\bf Remark.} The connectivity of $K$ is necessary in Theorem \ref{th4.6}. For example, if we let  $\{S_i\}_{i=1}^4$ be an IFS on $\mathbb R$ as follows:
$$
S_1(x)=\frac x4, \quad S_2(x)=\frac x4 + \frac 1{12}, \quad S_3(x)=\frac x4 + \frac 23, \quad S_4(x)=\frac x4 + \frac 34.
$$
Then $K = [0,1/3] \cup [2/3,1]$, and it is easy to check that the IFS satisfies the OSC (let $O = (0,1/3) \cup (2/3,1)$ as the open set). As $K$ consists of two intervals as connected components, we have $\beta_2^*=2$ and $\beta_3^* = \infty$ trivially.

\bigskip

In the rest of this section, we focus on the {\it post critically finite} (p.c.f.) self-similar sets \cite {Ki1}, and provide a criterion to determine $\beta_1^*$. We will need a general lemma as follow.
\medskip

\begin{lemma} \label{th4.8}
With the same assumption as in Theorem \ref{th4.1}, for a finite set $E \subset K$ with $\# E \geq 2$, if $R(\xi,\eta)>0$ for all distinct $\xi \not= \eta$ in $E$, then $R(\xi,E \setminus \{\xi\})>0$ for all $\xi \in E$.
\end{lemma}

\begin{proof}  We prove the lemma by induction on $\# E$.
It is trivial if $\#E =2$. Suppose the lemma holds for $\# E = m$ (m $\geq$ 2). Now let $\# E = m+1$. We choose arbitrarily three distinct points $\xi_1, \xi_2, \xi_3 \in E$. Then it suffices to show that $R(\xi_1,E \setminus \{\xi_1\})>0$. By induction hypothesis, we have three positive limiting resistances $R_1:=R(\xi_1, E \setminus \{\xi_1,\xi_2\})$, $R_2:=R(\xi_2, E \setminus \{\xi_2,\xi_3\})$ and $R_3:=R(\xi_3, E \setminus \{\xi_3,\xi_1\})$.

\vspace{1mm}

 For sufficiently large $n$, let $f_{1,n}$, $f_{2,n}$, $f_{3,n} \in \ell(X_n)$ be the energy minimizer in \eqref{eq4.7} with $(\Phi,\Psi)=(\{\xi_1\}, E \setminus \{\xi_1,\xi_2\})$, $(\{\xi_2\}, E \setminus \{\xi_2,\xi_3\})$, $(\{\xi_3\}, E \setminus \{\xi_3,\xi_1\})$ respectively. Fix a $\kappa$-sequence $\{\kappa_n\}_n$. Let $\gamma_{1,n} = f_{1,n}(\kappa_n(\xi_2))$, $\gamma_{2,n}=f_{2,n}(\kappa_n(\xi_3))$, and $\gamma_{3,n} = f_{3,n}(\kappa(\xi_1))$. Then $\gamma_{i,n} \in [0,1]$ for $i=1,2,3$. For sufficiently large $n$, we can check that the function
\begin{equation} \label{eq4.15}
h_n:=\frac{1}{1+\gamma_{1,n}\gamma_{2,n}\gamma_{3,n}}\big (f_{1,n}-
(\gamma_{1,n}f_{2,n}+\gamma_{1,n}\gamma_{2,n}f_{3,n})\big )
\end{equation}
satisfies $h_n(\kappa_n(\xi_1))=1$, and $h_n = 0$ on $\kappa_n(E \setminus \xi_1)$. Moreover, $h_n$ is harmonic on $X_n \setminus \kappa_n(E)$, thus $\en_{X_n}[h_n]=(R_n(\xi_1, E \setminus \{\xi_1\}))^{-1}$ by \eqref{eq4.3}. Hence
\begin{align*}
R(\xi_1, E \setminus \{\xi_1\}) &= \lim_{n \to \infty} (\en_{X_n}[h_n])^{-1} \\
& \geq \lim_{n \to \infty}\left(\frac{3(\en_{X_n}[f_{1,n}]+\gamma_{1,n}^2 \en_{X_n}[f_{2,n}]+\gamma_{1,n}^2 \gamma_{2,n}^2 \en_{X_n}[f_{3,n}])}{(1+\gamma_{1,n}\gamma_{2,n}\gamma_{3,n})^2}\right)^{-1} \\
& = \left(\frac{3(R_1^{-1}+\gamma_{1,n}^2 R_2^{-1}+\gamma_{1,n}^2 \gamma_{2,n}^2 R_3^{-1})}{(1+\gamma_{1,n}\gamma_{2,n}\gamma_{3,n})^2}\right)^{-1} > 0.
\end{align*}
This completes the proof of the induction.
\end{proof}

\medskip
Following Kigami \cite {Ki1}, for an IFS $\{S_j\}_{j=1}^N$ with a self-similar set $K$,   we let $C_K = \bigcup_{i,j \in \Sigma, i \neq j}(S_i(K) \cap S_j(K))$, and define a {\it critical set} by $\mathcal C = \pi^{-1}(C_K)$, a {\it post critical set} by $\mathcal P = \bigcup_{n \geq 1} \sigma^n(\mathcal C)$. We call $K$  {\it post critically finite }(p.c.f.) if ${\mathcal P}$ is a finite set.

\medskip

It is known that for the similitudes  $S_j = r_j (R_jx +b_j), j=1, \cdots , N$, if  the $\{R_j\}_{j=1}^N$  are commensurable, then the p.c.f.~property implies the OSC \cite {DL}, and the statement is not true without the commensurable assumption \cite {TKV}. We introduce two geometric conditions on the p.c.f.~sets:

\vspace {0.2cm}

\noindent {\rm (C)} \ {\it for any family of distinct subcells $S_{i_1}(K), \cdots, S_{i_k}(K)$ that intersects at a point $p$, there exists $0<\delta <1$ and closed cones $\mathcal C_j, 1\leq j \leq k$ with vertex at $p$ such that
$$
 S_{i_j}( K) \cap B(p, \delta) \subset \mathcal C_j, \ \ \hbox {and} \ \  \mathcal C_j \cap \mathcal C_\ell = \{p\} \quad \forall \ 1\leq j,\ \ell \leq k, \ j \not = \ell;
$$}

\vspace {0.1cm}
\noindent {\rm (H)}\ {\it there exists constant $\gamma>0$ such that for any $\xxx, \yyy \in X$ with $|\xxx|=|\yyy|$, if $S_{\xxx}(K) \cap S_{\yyy}(K)=\emptyset$, then}
$$
{\rm dist}(S_\xxx(K),S_\yyy(K))>\gamma \cdot r^{|\xxx|}.
$$
Condition (C) says the intersecting cells are separated by closed cones (except at the vertices), and the geometric meaning is clear. Condition (H) says that if two cells are disjoint, then they are ``strongly" separate; it has been used in \cite {Jo}, \cite {LW1} and \cite{GuL}. Note that the familiar self-similar sets satisfies this condition, and it is proved in \cite{GuL} the if the IFS is of the form  $S_j(x)= r(x+b_j)$ and is p.c.f., then $K$ satisfies condition (H).

\begin{lemma} \label{th4.10}
Let $K$ be a p.c.f.~self-similar set that satisfies either (C) or (H). Suppose for  $\alpha <\beta <\beta'$,  $u$ satisfies $u \circ S_i \in \Lambda_{2,2}^{\alpha,\beta'/2}$ for each $i \in \Sigma$, then $u \in \Lambda_{2,2}^{\alpha,\beta/2}$.
\end{lemma}

\begin{proof}
First suppose that $K$ satisfies {\rm (C)}. By the separation of the cones, and the cosine law of a triangle, we can show that there exists  $c>0$ such that if $S_i(K)$ intersects $S_j (K)$ at $p$, and for $\xi \in S_i(K) \cap B(p, \delta) , \  \eta \in S_j(K) \cap B(p, \delta)$,
\begin{equation} \label{eq4.18}
 |\xi -\eta| \geq c (|\xi- p| + |\eta -p|) \geq 2c |\xi-p|^{1/2}\cdot |\eta-p|^{1/2}.
\end{equation}
Since $u\circ S_i \in \Lambda_{2,2}^{\alpha,\beta'/2}$, it follows from Theorem \ref{th2.1} that $u\circ S_i \in C^{(\beta'-\alpha )/2}(K)$. As $u(\xi) = \sum_{i=1}^N u(\xi) \chi_{S_i(K)} (\xi)$, we show that $u$ is also H\"older continuous of order $(\beta'- \alpha )/2$ at any $p \in S_i(K) \cap S_j(K)$.  Indeed we observe that for $\xi \in S_i(K) \cap B(p, \delta) ,  \eta \in S_j(K) \cap B(p, \delta)$,
 \begin{align*}
 |u(\xi) - u(\eta)|&\leq |u(\xi)-u(p)| + |u(\eta)-u(p)|\\
 &\leq C( |\xi -p|^{ (\beta'-\alpha)/2} + |\eta -p|^{(\beta'- \alpha)/2})\\
 & \leq 2C (|\xi -p|+ |\eta -p|)^{(\beta'- \alpha)/2}\\
 & \leq C_1|\xi - \eta|^{(\beta'- \alpha)/2}  \qquad (\hbox {by} \ \eqref{eq4.18}).
\end{align*}
 This together with \eqref{eq4.18} imply
 \begin{align} \label{eq4.19}
 & \int_{S_i(K)\cap B(p, \delta)}\int_{S_j(K)\cap B(p, \delta)} \frac {|u(\xi) -u(\eta)|^2}{|\xi- \eta|^{\alpha + \beta}}d\nu(\xi) d\nu(\eta) \nonumber\\
 \leq C_2 & \int_{S_i(K)\cap B(p, \delta)}\frac {d\nu(\xi)}{|\xi- p|^{\alpha+\beta-\beta'}}
 \cdot \int_{S_j(K)\cap B(p, \delta)}\frac {d\nu(\eta)}{|\eta- p|^{\alpha+\beta-\beta'}} < \infty .
 \end{align}

Now as $u(\xi) = \sum_{i=1}^N u(\xi) \chi_{S_i(K)} (\xi)$, we have
 \begin{align*}
 \en_K[u] = &\sum_{i, j =1}^N \int_{S_i(K)}\int_{S_j(K)} \frac {|u(\xi) - u(\eta)|^2} {|\xi- \eta|^{\alpha + \beta}}d\nu(\xi) d\nu(\eta) \\
 = & \big (\sum_{i=j} +  \sum_{i\not = j} \big ) \int_{S_i(K)}\int_{S_j(K)} \frac {|u(\xi) - u(\eta)|^2} {|\xi- \eta|^{\alpha + \beta}}d\nu(\xi) d\nu(\eta)\\
 := & S_I + S_{II}.
 \end{align*}
By a change of variable,
$$
S_I = \sum_{i=1}^n r_i^{\alpha -\beta} \int_K \int_K \frac {|u\circ S_i(\xi) - u\circ S_i(\eta)|^2} {|\xi- \eta|^{\alpha + \beta}}d\nu(\xi) d\nu(\eta) < \infty.
$$
By \eqref{eq4.19}, it is easy to check that $S_{II} < \infty$. This shows that $\en_K[u] < \infty$, so that $u \in {\mathcal D}_K = \Lambda_{2,2}^{\alpha,\beta/2}$.

\vspace{2mm}

Next we suppose that $K$ satisfies {\rm (H)}. Assume without loss of generality that ${\rm diam}(K)=1$. For $p \in S_i(K)\cap S_j(K)$, $i \neq j \in \Sigma$, let $\delta=\frac 12 \min\{|p-q|:q \in S_i(K)\cap S_j(K),q \neq p\}$. Following the same argument in the last paragraph, it suffices to show that \eqref{eq4.18} holds for $\xi \in S_i(K) \cap B(p, \delta)$ and $\eta \in S_j(K) \cap B(p, \delta)$. Indeed, suppose that $|\eta-p| \leq |\xi-p| \in (r^k, r^{k-1}]$ for some positive integer $k$. Let $\xxx, \yyy \in \mathcal J_k$ with $\xi \in S_\xxx(K) \subset S_i(K)$ and $\eta \in S_\yyy(K) \subset S_j(K)$. As ${\rm diam} (S_\xxx (K)), {\rm diam} (S_\yyy(K)) \leq r^k$, $S_\xxx(K) \cap S_\yyy(K) = \emptyset$. Hence by condition {\rm (H)},
$$
|\xi-\eta| \geq \gamma\cdot r^k \geq \gamma r |\xi-p| \geq \frac{\gamma r}{2} (|\xi-p|+|\eta-p|).
$$
This completes the proof.
\end{proof}

\medskip

We let  $V_0 = \pi(\mathcal P)$ be the ``boundary" of a p.c.f.~set $K$, and let $V_n = \cup_{\xxx \in \Sigma^n} S_\xxx(V_0)$, $n \geq 1$.

\begin{theorem} \label{th4.9}
With the same assumption as in Theorem \ref{th4.1}, assume further $K$ is a p.c.f.~set with boundary $V_0$ and satisfies {\rm (C)} or {\rm (H)}.  Suppose
\begin{equation} \label{eq4.16}
R^{(\lambda-\varepsilon)}(\xi, \eta)>0, \qquad \forall \, \xi \neq \eta \in V_0,
\end{equation}
for some $\varepsilon \in (0,\lambda)$, then  $\dom_K$ is dense in $C(K)$ (with the supremum norm).

\vspace{0.1cm}
Consequently, $\beta_1^* = \log \lambda_1^* / \log r$ if
\begin{equation} \label{eq4.17}
\lambda_1^*:= \inf \{\lambda>0: R^{(\lambda)}(\xi, \eta)>0, \ \forall \, \xi \neq \eta \in V_0\} \in (0, r^\alpha),
\end{equation}
and   $\beta_1^* \leq \alpha$ otherwise.
\end{theorem}

\medskip

\begin{proof}
Let $V_0 = \{\xi_1,\xi_2,\cdots, \xi_m\}$. If $R(\xi_i,\xi_j)=0$ for some $i \neq j$, then $\mathcal D_K$ is not dense in $C(K)$ by Corollary \ref{th4.3}. Now suppose that \eqref{eq4.16} holds and let $\beta_0=\log (\lambda-\varepsilon)/\log r$. Then $R^{(\lambda-\varepsilon)}(\xi_i, V_0 \setminus \{\xi_i\})>0$ for all $i$ by Lemma \ref{th4.8}. Thus we can obtain a ``basis" of functions $\{u_i\}_{1 \leq i \leq m} \subset \Lambda_{2,2}^{\alpha,\beta_0/2}$ with $u_i(\xi_j) = \delta_{ij}$ following from Proposition \ref{th4.2}. Using  linear combinations, for any $v \in \ell(V_0)$, one can check that $u = \sum_{i=1}^m v(\xi_i)u_i \in \Lambda_{2,2}^{\alpha,\beta_0/2} \subset \mathcal D_K$ satisfies $u|_{V_0} = v$.

\vspace{0.1cm}

We use induction on $n$ to claim that for $\beta_n=\log(\lambda-\frac{\varepsilon}{2^n})/\log r$ and for any $v \in \ell(V_n)$, there exists $u \in \Lambda_{2,2}^{\alpha,\beta_n/2} \subset \mathcal D_K$ such that $u|_{V_n}=v$. Indeed, the above verifies the case $n=0$. Assume the statement holds for some $n$. Let $v \in \ell(V_{n+1})$. Note that $V_n=S_i^{-1}(V_{n+1} \cap S_i(K))$ for all $i \in \Sigma$. By induction hypothesis, for each $i$, there exists $w_i \in \Lambda_{2,2}^{\alpha,\beta_n/2}$ such that $w_i|_{V_n} =v|_{V_{n+1} \cap S_i(K)} \circ S_i$. Let $u(\xi)=\sum_{i=1}^N (w_i \circ S_i^{-1}) (\xi) \chi_{S_i(K)}(\xi)$. Then $u|_{V_{n+1}}=v$ and $u \circ S_i = w_i \in \Lambda_{2,2}^{\alpha,\beta_n/2}$. By Lemma \ref{th4.10}, $u \in \Lambda_{2,2}^{\alpha,\beta_{n+1}/2} \subset \mathcal D_K$. This completes the proof of induction.

\vspace{1mm}

As $n$ tends to infinity, $\beta_n$ decreases to $\beta=\log \lambda / \log r$, and $\bigcup_{n \geq 0} V_n$ is dense in $K$. Hence $\dom_K = \Lambda_{2,2}^{\alpha,\beta/2}$ is dense in $C(K)$.
\end{proof}

\bigskip

\section{Network reduction and examples}
\label{sec:6}

\noindent In this section, we will provide a device to calculate the limiting resistances and the critical exponents of the Besov spaces on $K$. We first recall some formal notions and techniques on electric network theory \cite {DS,LP}.

\medskip

Let  $\mathcal N=(V,c)$ denote the {\it (electric) network} with vertex set $V$ (finite or countably infinite) and conductance $c: V \times V \to [0, \infty)$ ($c(x,y)=c(y,x)$ for all $x,y \in V$). The edge set $E = \{(x,y) \in (V \times V) \setminus \Delta: c(x,y)>0\}$. An edge $(x,y) \in E$ is referred as a {\it resistor} (or {\it conductor}) with resistance $r_{xy}=r(x,y)=c(x,y)^{-1}$. The energy of $f \in \ell(V)$ on $\mathcal N$ is given  by
\begin{equation} \label{eq5.01}
\mathcal E_{\mathcal N}[f] = \frac 12 \sum_{x,y \in V} c(x,y)(f(x)-f(y))^2
\end{equation}
as in \eqref{eq2.2}. Also we can define the effective resistance $R_{\mathcal N}(A,B)$ between two nonempty subsets $A,B \subset V$ as in \eqref{eq4.1}.

\medskip

\begin{definition} \label{de5.1}
For two networks $\mathcal N_1=(V_1,c_1)$ and $\mathcal N_2=(V_2,c_2)$ with a set of common vertices $U \subset V_1 \cap V_2$, $\# U \geq 2$, we say that $\mathcal N_1$ and $\mathcal N_2$ are {\rm equivalent} on $U$ if for any $f \in \ell(U)$,
\begin{equation} \label{eq5.02}
\inf\{\mathcal E_{\mathcal N_1}[g_1]: g_1 \in \ell(V_1), g_1 |_U=f\} = \inf\{\mathcal E_{\mathcal N_2}[g_2]: g_2 \in \ell(V_2), g_2 |_U=f\}.
\end{equation}
\end{definition}

\medskip

It is easy to show that if $\mathcal N_1$ and $\mathcal N_2$ are equivalent on $U$, then they are also equivalent on any $U' \subset U$. As a result, $R_{\mathcal N_1}(A,B) = R_{\mathcal N_2}(A,B)$ for any nonempty $A,B \subset U$.

\medskip

The two most basic transformations to reduce networks to equivalent ones are the  {\it series law} and  the {\it parallel law} of resistance.
The third one is  the {\it $\Delta$-Y transform} (or {\it star-triangle Law}): let $\mathcal N_1$ be the triangle shaped network with $V_1=\{x,y,z\}$ as on the left of Figure \ref{fig:1}, and let $\mathcal N_2$ be the starlike network on the right with $V_2 = V_1 \cup \{p\}$;  for the two network to be equivalent, the resistances are related by
\begin{figure}[ht]
\begin{center}
\includegraphics[width=72mm,height=25mm]{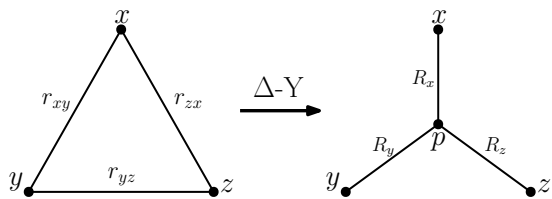}
\caption{{\small $\Delta$-Y transform}}
\label{fig:1}
\end{center}
\end{figure}
$$
R_x=\frac{r_{xy}r_{zx}}{r_{xy}+r_{yz}+r_{zx}}, \quad R_y=\frac{r_{xy}r_{yz}}{r_{xy}+r_{yz}+r_{zx}}, \quad R_z=\frac{r_{zx}r_{yz}}{r_{xy}+r_{yz}+r_{zx}}
$$
respectively. For some network $\mathcal N=\{V,c\}$, $\# V > 3$ with proper symmetry, we can add one vertex and transform it to an equivalent starlike network (see the examples in the sequel and \cite{K} for more details); we regard such transformation as a {\it generalized} $\Delta$-Y transform.

\medskip

More generally, we have from Proposition \ref{th4.0}, that if $V = V^\circ \cup \partial V$, $\# \partial V \geq 2$ then for $f \in \ell(\partial V)$,
\begin{equation} \label{eq6.3}
\min \{\mathcal E_{\mathcal N}[g]: g \in \ell(V), g |_{\partial V}=f\} = \frac 12 \sum_{x,y \in \partial V, x \neq y} c_\ast(x,y)(f(x)-f(y))^2.
\end{equation}
Then the network $\mathcal N_\ast=\{\partial V, c_\ast\}$ is equivalent to $\mathcal N$ on $\partial V$. For proper $\partial V$, the graph of network $\mathcal N_\ast$ may contain a complete subgraph $K_n$. In this case, we say that the transform $\mathcal N \to \mathcal N_\ast$ is a {\it local completion}. For example, as in Figure \ref{fig:2}, let $\partial V=\{x_1,x_2,\ldots,x_5\}$, then the graph of $\mathcal N_\ast$ is a complete graph $K_5$.

\begin{figure}[ht]
\begin{center}
\includegraphics[width=100mm,height=30mm]{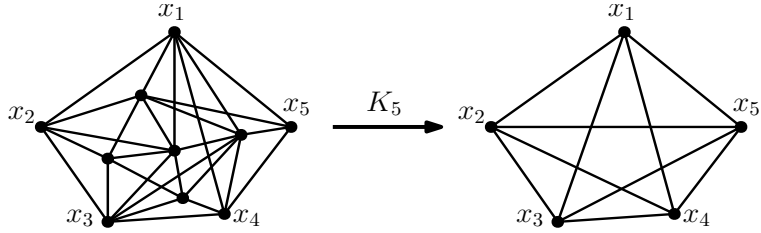}
\caption{{\small Local completion}}
\label{fig:2}
\end{center}
\end{figure}

Besides the above mentioned transformations, there are other  basic tools in network reduction we will use: cutting and shorting, and the Rayleigh's monotonicity law, namely, {\it if some resistances of resistors in a network are increased (decreased), then the effective resistance between any two points in the graph can only increase (decrease)}.

\vspace{5mm}

\begin{example} \label{ex5.1} {\bf Cantor middle third set}
Let $S_1(\xi)=\frac 13\xi $ and $S_2(\xi)=\frac 13(\xi+2)$ on $\mathbb R$. Then the self-similar set $K$ is the {\it Cantor middle-third set} with ratio $r = \frac 13$. It is totally disconnected  and the Hausdorff dimension is $\alpha = \frac{\log 2}{\log 3}$. The critical exponents $\beta_1^*=\beta_2^*=\beta_3^*=\infty$.
\end{example}
\begin{figure}[ht]
\begin{center}
\includegraphics[width=50mm,height=25mm]{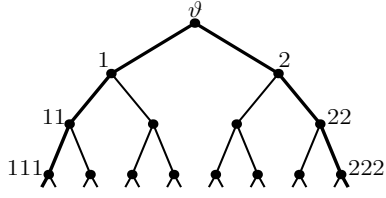}
\caption{{\small The limiting resistance for Cantor set}}
\label{fig:3}
\end{center}
\end{figure}

\noindent Indeed, for $\lambda \in (0,\frac 12)$ ($r^\alpha = \frac 12$), the limiting resistance between 0 and 1 (see Figure \ref{fig:3}) is
\begin{align*}
R^{(\lambda)}(0,1) &= R^{(\lambda)}(1^\infty,2^\infty) = \lim_{n \to \infty} R^{(\lambda)}_n (1^n,2^n) \\
&= \lim_{n \to \infty}\left(\sum_{k=1}^n c(1^k,1^{k-1})^{-1} + \sum_{k=1}^n c(2^k,2^{k-1})^{-1}\right) \\
& = 2 \lim_{n \to \infty}\sum_{k=1}^n (2\lambda)^{-k}= \frac{4\lambda}{1-2\lambda},
\end{align*}
and Theorem \ref{th4.9} implies the result. \hfill $\Box$

\medskip

\begin{example} \label{ex5.4} {\bf Sierpi{\'n}ski gasket}  It is the self-similar set $K$ generated by the maps $S_i(\xi)=\frac 12 (\xi - e_{i-1})+e_{i-1}$ where  $e_0=0$ and $e_i, i=1,\ldots,N-1$ are the standard basis vectors in $\mathbb{R}^{N-1}$.  It is a p.c.f.~set with $\mathcal P=\{1^\infty, 2^\infty, \ldots, N^\infty\}$, and  $\alpha = \dim_H K = \frac{\log N}{\log 2}$. For the $\lambda$-NRW ($r^\alpha = \frac 1N$),  the conductance is $c(\xxx,\xxx^-) = c(\xxx, \yyy) = (\lambda N)^{-|\xxx|}$ where $\xxx\sim_h\yyy$.
The critical exponent of $\Lambda^{\alpha, \beta/2}_{2,2}$ is
$$ \beta_1^*= \beta^*_2 = \beta^*_3 = \frac{\log (N+2)}{\log 2} \quad \hbox{at} \quad \lambda = \frac 1{N+2}.$$
(The critical exponent is known in \cite {Jo}.)
\end{example}

\begin{figure}[ht]
\begin{center}
\includegraphics[width=92mm,height=35mm]{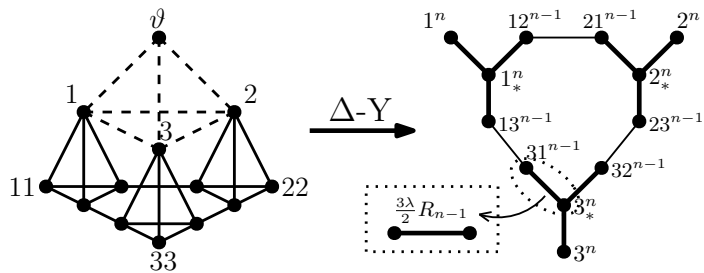}
\caption{{\small Cutting in Sierpi{\'n}ski gasket, $N=3$}}
\label{fig:8}
\end{center}
\end{figure}

We only prove the case $N=3$ (the other cases are quite similar; the reader is also advised to use $N=2$ to get a clearer picture). By symmetry, it suffices to find the limiting resistance $R^{(\lambda)}(1^\infty,2^\infty)$.  We denote $R^{(\lambda)}_n=R^{(\lambda)}_n(1^n,2^n)$ for short.

\medskip

To estimate the upper bound,  we delete the edges $(\vartheta,i)$, $(ij^k,ji^k)$, for $i \neq j \in \Sigma$, $k=0,1,\ldots, n-2$ in the subgraph of $X_n$ (see Figure \ref{fig:8}). Then we get a new subgraph consisting of $3$ copies of $X_{n-1}$ with $3$ horizontal edges $(ij^{n-1},ji^{n-1})$, $i \neq j \in \Sigma$ at level $n$ connecting them; we label these copies by $1,2,3$ such that the copy $i$ contains the vertex $i^n$. Then apply the the  $\Delta$-Y transform to the three vertices in $\mathcal A_i:=\{ij^{n-1}:  j \in \Sigma\}$ at the $n$-th level of each copy to get a starlike tree with center $i^n_*$, $i \in \Sigma$ respectively.  As the resistance between any pair of vertices in $\mathcal A_i$ equals $3\lambda R_{n-1}$, it follows that the resistance between $i^n_*$ and a vertex in $\mathcal A_i$ in the corresponding starlike tree is $\frac{3\lambda}{2}R_{n-1}$. Moreover, between any pair $i^n_*, j^n_*$, $i \neq j$, there is a 3-step path $[i^n_*, ij^{n-1}, ji^{n-1}, j^n_*]$. Replacing these paths with resistors, we get a triangle with vertices $\{i^n_*: i \in \Sigma\}$ and each side has resistance $3\lambda R_{n-1}+(3\lambda)^n$.

By applying the monotonicity law and the series law,
\begin{align*}
R^{(\lambda)}_n &\leq R(1^n,1^n_*)+R(1^n_*,2^n_*)+R(2^n_*,2^n) \\
&= \frac{3\lambda}{2}R^{(\lambda)}_{n-1}+\frac 23(3\lambda R^{(\lambda)}_{n-1}+(3\lambda)^n)+\frac{3\lambda}{2}R^{(\lambda)}_{n-1} \\
&= 5\lambda R^{(\lambda)}_{n-1} + 2\cdot3^{n-1}\lambda^n.
\end{align*}
Hence $R^{(\lambda)}(1^\infty,2^\infty)=\lim_{n \to \infty} R^{(\lambda)}_n = 0$ for $\lambda \in (0,\frac 1{5})$. By Proposition \ref{th4.3} and Theorem \ref{th4.4}, we have $\beta_1^* \leq \beta_3^* \leq \frac{\log 5}{\log 2}$.

\vspace{2mm}

 To obtain the lower bound of the critical exponent, we need another technique. We reassign the conductance on the $n$-th level of the subgraph $X_n$ ($n \geq 1$): for $\mu>0$, let $\widetilde c(\xxx,\xxx^-) = (3\lambda)^{-|\xxx|}$ for $\xxx \in X_n$, and let
$$
\widetilde c(\xxx,\yyy) =
\begin{cases}
\ (3\lambda)^{-|\xxx|}, & \quad \hbox{if } |\xxx|<n, \\
\ \mu^{-1}(3\lambda)^{-n}, & \quad \hbox{if } |\xxx|=n,
\end{cases}
\qquad \hbox{for } \xxx \sim_h \yyy \in X_n.
$$
Denote the level-$n$ resistance between $1^n$ and $2^n$ with respect to the above $\widetilde c$ by $R_n^{(\lambda,\mu)}$. Then apply the generalized $\Delta$-Y transforms to each triangle $(\xxx,\xxx 1,\xxx 2, \xxx 3)$ for $\xxx \in \mathcal J_{n-1}$, and then replace each pair $\{\xxx,\xxx'\}$ by a single $\xxx$ (see Figure \ref{fig:5} for $N=2$ for a clearer illustration; Figure \ref{fig:9} for $N=3$ corresponds to the dotted box in Figure \ref{fig:5}).
\begin{figure}[ht]
\begin{center}
\includegraphics[width=90mm,height=55mm]{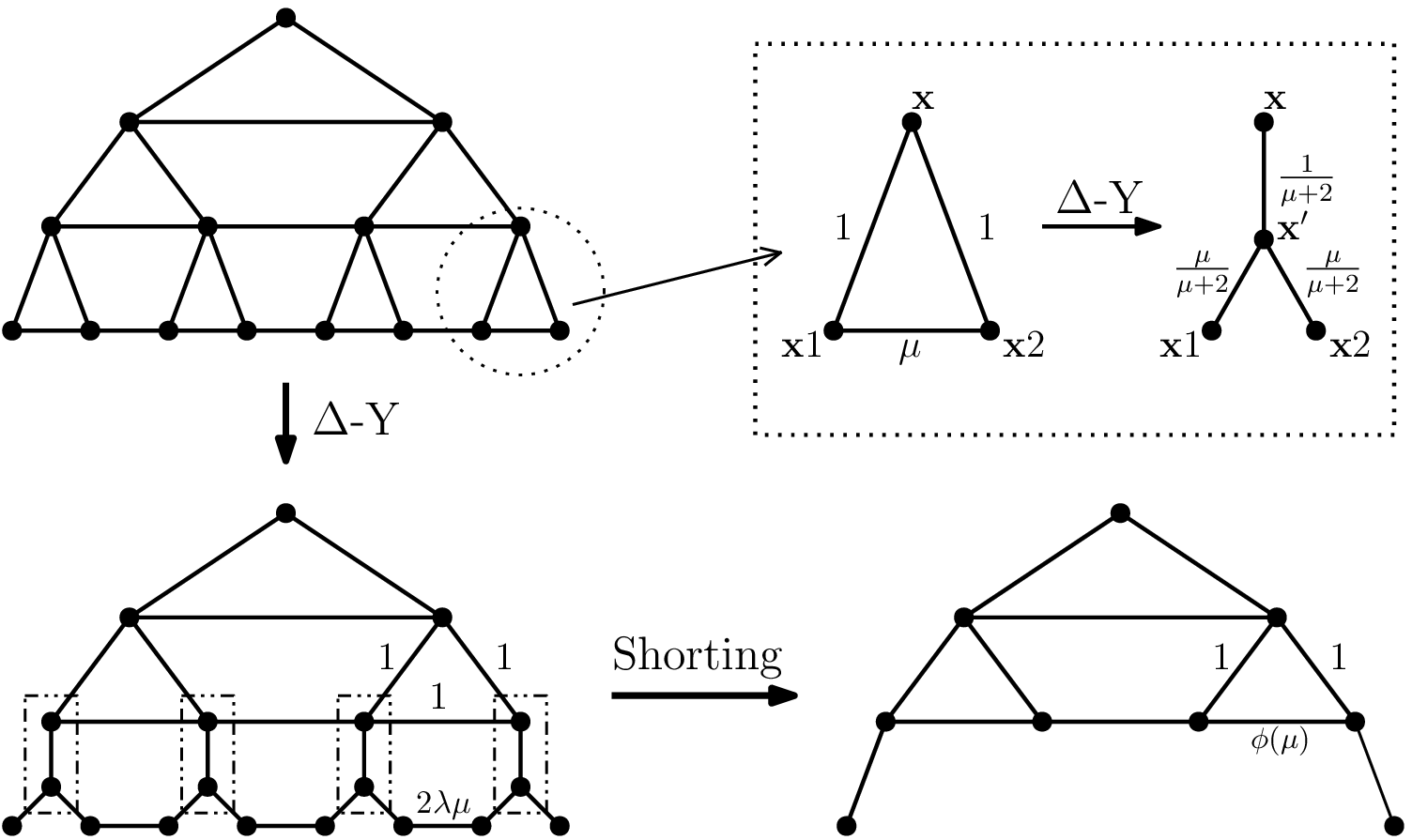}
\caption{{\small $\mu$-parameter and shorting for $N=2$
}} \label{fig:5}
\end{center}
\end{figure}
\begin{figure}[ht]
\begin{center}
\includegraphics[width=80mm,height=25mm]{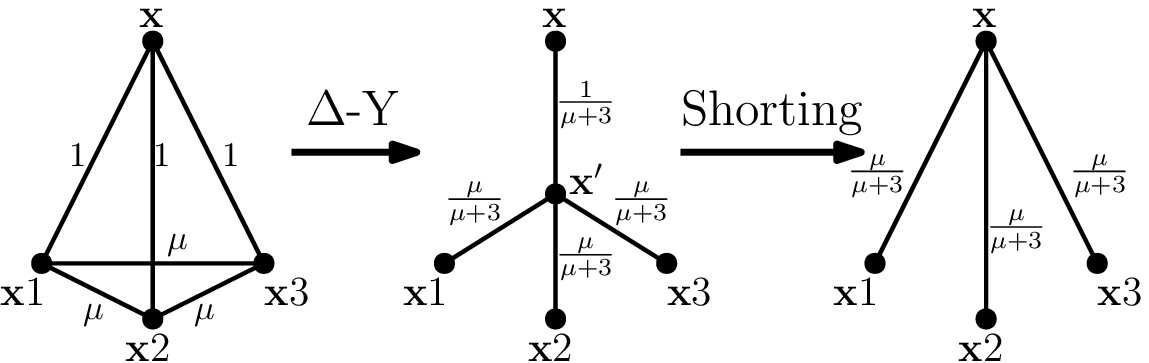}
\caption{{\small $\mu$-parameter and shorting for $N=3$}}
\label{fig:9}
\end{center}
\end{figure}

\medskip

\noindent We have
\begin{equation} \label{eq5.6}
R^{(\lambda,\mu)}_n \geq \frac{2\mu}{\mu+3}(3\lambda)^n + R^{(\lambda,\phi(\mu))}_{n-1},
\end{equation}
where $\phi$ is given by the parallel resistance formula
\begin{equation} \label{eq5.7}
\phi(\mu)^{-1} = \left[3\lambda\left(\frac{2\mu}{\mu+3}+\mu\right)\right]^{-1} + 1.
\end{equation}
The equation $\phi(\mu)=\mu$ has a solution $\overline \mu \in (0,1)$ if and only if $\lambda>\frac 1{5}$. With such fixed point $\overline \mu$, by \eqref{eq5.6}, we have
$R^{(\lambda)}(1^\infty,2^\infty) \geq \lim_{n \to \infty} R^{(\lambda,\overline \mu)}_n \geq R^{(\lambda,\overline \mu)}_1 > 0$. By Theorem \ref{th4.8}, we have $\frac{\log 5}{\log 2} \leq \beta_1^* \leq \beta_3^*$, and completes the proof. \hfill $\Box$

\bigskip

In the next example, we adjust the above method slightly for the new situation with two different limiting resistances of $(i^\infty, j^\infty)$.

\medskip

\begin{example} \label{ex5.5} {\bf Pentagasket}
The pentagasket is the attractor $K$ of the five similitudes  $S_i(\xi)=\frac{3-\sqrt 5}{2}(\xi-p_i)+p_i$, here we identify  $\mathbb R^2$ with $\mathbb C$, and $p_i = e^{2\pi i /5}$. It is a p.c.f.~set with $\mathcal P = \{1^\infty, 2^\infty, \ldots , 5^\infty\}$, and $\alpha := \dim_H K = -\frac{\log 5}{\log ((3-\sqrt 5)/2)}$. As $r^\alpha = \frac 15$, the $\lambda$-NRW has conductance $(5\lambda)^{-n}$ on level $n$. The critical exponent is
$$
\beta_1^* =   \beta_2^* = \beta_3^* =\frac{\log ((\sqrt{161}-9)/40)}{\log((3-\sqrt 5)/ 2)}   \quad \hbox {at} \quad \lambda = \frac {\sqrt{161}-9}{40}.
$$
 (The critical exponent is known in \cite{Ku}.)
\end{example}

\begin{figure}[ht]
\begin{center}
\includegraphics[width=90mm,height=40mm]{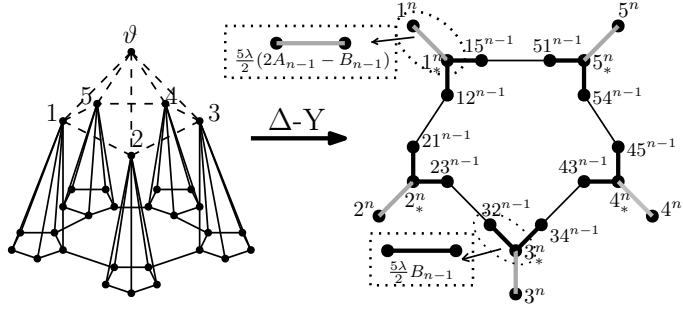}
\caption{{\small Cutting in pentagasket}} \label{fig:10}
\end{center}
\end{figure}

To determine the critical exponent,  we need to calculate the limiting resistances
$$
R^{(\lambda)}(1^\infty,2^\infty) \quad \hbox {and} \quad R^{(\lambda)}(1^\infty,3^\infty).
$$
We denote $A_n = R^{(\lambda)}_n(1^n,3^n)$ and $B_n = R^{(\lambda)}_n(1^n,2^n)$ for short.
By referring to Figure \ref{fig:10}, and using the same technique as before, we have
\begin{align*}
A_n &\leq R(1^n,1^n_*)+R(1^n_*,3^n_*)+R(3^n_*,3^n) \\
&= 5\lambda(2A_{n-1}-B_{n-1})+\left[(10\lambda B_{n-1}+2(5\lambda)^n)^{-1}+(15\lambda B_{n-1}+3(5\lambda)^n)^{-1}\right]^{-1} \\
&= 10\lambda A_{n-1}+\lambda B_{n-1}+\frac 65(5\lambda)^n.
\end{align*}
Analogously, we have $B_n \leq 10\lambda A_{n-1}-\lambda B_n+\frac 45 (5\lambda)^n$. As the coefficient matrix $\left(
\begin{array}{cc}
10\lambda & \lambda \\
10\lambda & -\lambda \\
\end{array}
\right)$ has eigenvalues $\frac{9\pm\sqrt{161}}{2}\lambda$, we have $\lim\limits_{n \to \infty} A_n = \lim\limits_{n \to \infty} B_n = 0$ if $\lambda < (\frac{9+\sqrt{161}}{2})^{-1} = \frac{\sqrt{161}-9}{40}$.
Hence $R^{(\lambda)}(1^\infty,2^\infty)=R^{(\lambda)}(1^\infty,3^\infty) = 0$ for $\lambda \in (0,\frac{\sqrt{161}-9}{40})$. By Proposition \ref{th4.3} and Theorem \ref{th4.4}, we have $\beta_1^* \leq \beta_3^* \leq \frac{\log ((\sqrt{161}-9)/40)}{\log((3-\sqrt 5)/ 2)}$.

\vspace{2mm}

 To obtain the lower bound of the critical exponents, we reassign the conductance on the bottom of the subgraph $X_n$ ($n \geq 1$) with two parameters $\mu_1$ and $\mu_2$
: for $\mu_1,\mu_2 \in (0,1)$, let $\widetilde c(\xxx,\xxx^-) = (5\lambda)^{-|\xxx|}$ for $\xxx \in X_n$, and let
$$
\widetilde c(\xxx,\yyy) =
\begin{cases}
\ (5\lambda)^{-|\xxx|}, & \quad \hbox{if } |\xxx|<n, \\
\ \mu_1^{-1}(5\lambda)^{-n}, & \quad \hbox{if } |\xxx|=n \hbox{ and } \xxx^-=\yyy^-, \\
\ \mu_2^{-1}(5\lambda)^{-n}, & \quad \hbox{if } |\xxx|=n \hbox{ and } \xxx^-\neq \yyy^-,
\end{cases}
\qquad \hbox{for } \xxx \sim_h \yyy \in X_n.
$$
Denote the level-$n$ resistance between $1^n$ and $3^n$ (or $2^n$) with respect to above $\widetilde c$ by $A_n^{(\mu_1,\mu_2)}$ (or $B_n^{(\mu_1,\mu_2)}$).  We apply the local completion to each cone $(\xxx 1,\xxx 11,\xxx 13, \xxx 14)$, $(\xxx 2, \xxx 22, \xxx 24, \xxx 25)$, $(\xxx 3,\xxx 33, \xxx 35,\xxx 31)$, $(\xxx 4, \xxx 44, \xxx 41, \xxx 42)$, $(\xxx 5,\xxx 55, \xxx 52, \xxx 53)$ for $\xxx \in \mathcal J_{n-2}$, and then replace each complete subgraph $K_4$ by a starlike network with greater energy (Figure \ref{fig:11}).
\begin{figure}[ht]
\begin{center}
\includegraphics[width=85mm,height=25mm]{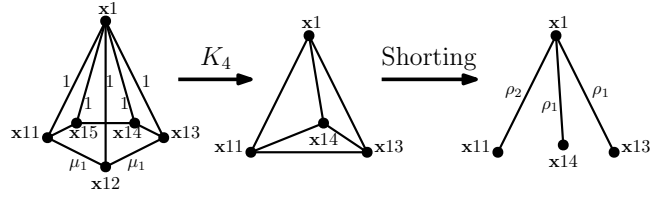}
\caption{{\small Shorting in pentagasket}}\label{fig:11}
\end{center}
\end{figure}
By a direct calculation, the conductance $c_\ast$ in $K_4$ is given by
$$
c_\ast(\xxx 1,\xxx 11)=\frac{\mu_1+4}{\mu_1+2}, \quad c_\ast(\xxx 1,\xxx 13)=c_\ast(\xxx 1,\xxx 14)=\frac{\mu_1+3}{\mu_1+2},
$$
$$
c_\ast(\xxx 11,\xxx 13)=c_\ast(\xxx 11,\xxx 14)=\frac{1}{\mu_1(\mu_1+2)}, \quad c_\ast(\xxx 13, \xxx 14)=\frac{1}{\mu_1},
$$
and the resistances in the star are given by
$$
\rho_1=\frac{\mu_1(\mu_1+2)}{\mu_1^2+5\mu_1+5}, \quad \rho_2=\frac{\mu_1(\mu_1+2)^2}{(\mu_1+1)(\mu_1^2+5\mu_1+5)}.
$$
By the monotonicity law and series law,
\begin{equation} \label{eq5.8}
A^{(\mu_1,\mu_2)}_n \geq 2\rho_2(5\lambda)^n + A^{(\phi_1(\mu_1,\mu_2),\phi_2(\mu_1,\mu_2))}_{n-1},
\end{equation}
(same inequality holds if we replace $A$ by $B$)
where $\phi_1$ and $\phi_2$ are given by the parallel resistance formulas
\begin{equation} \label{eq5.9}
\begin{cases}
\ \phi_1(\mu_1,\mu_2)^{-1} = \left[5\lambda\left(2\rho_1+\mu_2\right)\right]^{-1} + 1, \\
\ \phi_2(\mu_1,\mu_2)^{-1} = \left[5\lambda\left(2\rho_2+\mu_2\right)\right]^{-1} + 1.
\end{cases}
\end{equation}
The equations $\phi_i(\mu_1,\mu_2)=\mu_i$, $i=1,2$ have a solution $(\overline \mu_1, \overline \mu_2) \in (0,1)^2$ if and only if $\lambda>\frac{\sqrt{161}-9}{40}$. With such fixed point $(\overline \mu_1, \overline \mu_2)$, by \eqref{eq5.8}, we have
$R^{(\lambda)}(1^\infty,3^\infty) \geq \lim_{n \to \infty} A^{(\overline \mu_1, \overline \mu_2)}_n \geq A^{(\overline \mu_1, \overline \mu_2)}_1 > 0$. Similarly we also have $R^{(\lambda)}(1^\infty,2^\infty) > 0$ if $\lambda>\frac{\sqrt{161}-9}{40}$. By Theorem \ref{th4.8}, we have $\frac{\log ((\sqrt{161}-9)/40)}{\log((3-\sqrt 5)/ 2)} \leq \beta_1^* \leq \beta_3^*$.  \hfill  $\Box$

\bigskip

More computational issues on the critical exponent of nested fractals can be found in \cite{K}. Finally, we give an example that $\beta^*_1 \not = \beta^*_3$.

\begin{example} \label{ex5.3} {\bf Cantor set$\times$interval} \  Let $\Sigma=\{1,2,3,4,5,6\}$ and let $p_1=0, p_2=(0,\frac 13), p_3=(0,\frac 23), p_4=(\frac 23,0), p_5=(\frac 23,\frac 13), p_6=(\frac 23,\frac 23)$ in $\mathbb R^2$. For $i \in \Sigma$, let $S_i(\xi) = \frac 13\xi + p_i$ on $\mathbb R^2$. Then the self-similar set $K$ is the product of a Cantor middle-third set and a unit interval (see the associated augmented tree in Figure \ref{fig:6}), and $\alpha=\dim_H K=\frac {\log 2}{\log 3}+1=\frac {\log 6}{\log 3}$.   The $\lambda$-NRW has conductance $(6\lambda)^{-n}$ on the $n$-th level ($r^\alpha=\frac 16$). The critical exponents are
$$
\beta^*_1 = 2 \  \ \hbox {at} \ \  \lambda = \frac 19;  \qquad \beta^*_2 = \beta^*_3 =\infty
$$

\begin{figure}[ht]
\begin{center}
\includegraphics[width=80mm,height=28mm]{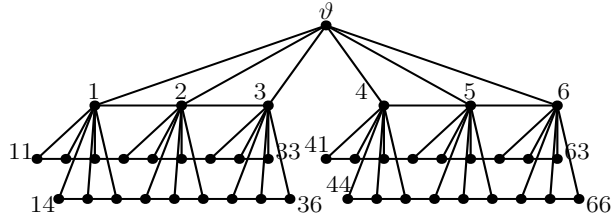}
\caption{{\small The graph for Cantor set$\times$interval}}\label{fig:6}
\end{center}
\end{figure}

{\rm First we show that $R^{(\lambda)}(1^\infty,4^\infty)>0$ for any $\lambda>0$. For $n \geq 1$, consider a function $f_n$ on $X_n$ defined by
$$
f_n(\xxx) = \begin{cases}
\ 1/2, & \quad \hbox{if } \xxx = \vartheta, \\
\ 1, & \quad \hbox{if } i_1=1,2,3, \\
\ 0, & \quad \hbox{if } i_1=4,5,6,
\end{cases}
\qquad \hbox{for } \xxx=i_1 i_2 \cdots i_k \in X_n.
$$
Then by \eqref{eq4.1}, $R^{(\lambda)}_n(1^n,4^n) \geq (\mathcal E_{X_n}[f_n])^{-1} = (6\cdot(\frac 12)^2\cdot\frac 1{6\lambda})^{-1} = 4\lambda$. Thus for any $\lambda>0$, $R^{(\lambda)}(1^\infty,4^\infty) = \lim_{n \to \infty}R^{(\lambda)}_n(1^n,4^n) \geq 4\lambda > 0$. By Theorem \ref{th4.4}, we have $\beta_3^*=\infty$. Also it is easy to see that $\beta_2^*=\infty$.

\vspace{2mm}
\begin{figure}[ht]
\begin{center}
\includegraphics[width=85mm,height=50mm]{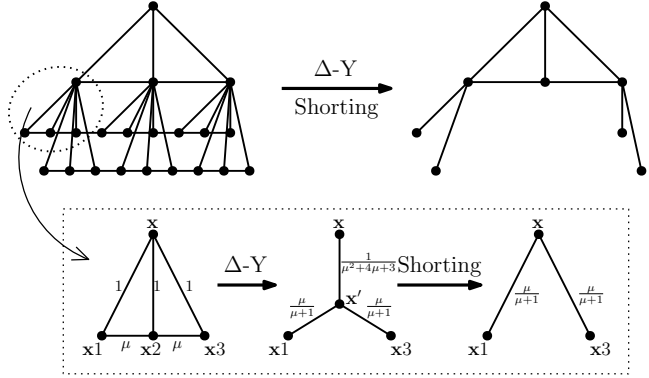}
\caption{{\small Shorting in Cantor set$\times$interval}}\label{fig:7}
\end{center}
\end{figure}

 Next we consider the limiting resistance $R^{(\lambda)}(1^\infty,3^\infty)$ by using a  similar shorting device as in previous examples. Denote $R^{(\lambda)}_n = R^{(\lambda)}_n(1^n,3^n)$ for short. As in Example \ref{ex5.5}, we reassign the conductance on the bottom of the subgraph $X_n$ by an additional factor $\mu^{-1}$, and by the same method applied to triangles $(\xxx,\xxx 1,\xxx 3)$ (also to  $(\xxx, \xxx 4,\xxx 6)$, see Figure \ref{fig:7}), we have
\begin{equation} \label{eq5.3}
R^{(\lambda,\mu)}_n \geq \frac{2\mu}{\mu+1}(6\lambda)^n + R^{(\lambda,\phi(\mu))}_{n-1},
\end{equation}
where $\phi$ is given by
\begin{equation} \label{eq5.4}
\phi(\mu)^{-1} = 2\left[6\lambda\left(\frac{2\mu}{\mu+1}+\mu\right)\right]^{-1} + 1.
\end{equation}
The equation $\phi(\mu)=\mu$ has a solution $\overline \mu \in (0,1)$ if and only if $\lambda>\frac 19$.

\medskip

With such fixed point $\overline \mu$, by \eqref{eq5.3}, we have
$R^{(\lambda)}(1^\infty,3^\infty) \geq \lim_{n \to \infty} R^{(\lambda,\overline \mu)}_n \geq R^{(\lambda,\overline \mu)}_1 > 0$.  On the other hand, we show that if $R^{(\lambda)}(1^\infty,3^\infty)>0$, then  $\lambda \geq \frac 19$. Without loss of generality, we assume that $0< \lambda < 1/6$. For $n \geq 1$, let $f_n$ be the energy minimizer (harmonic function) on $X_n$ with boundary conditions $f_n(1^n)=1$ and $f_n(3^n)=0$. Then $R_n(1^n,3^n) = \mathcal E_{X_n}[f_n]^{-1}$. By Corollary \ref{th4.3'} (iv) $\Rightarrow$ (iii), let $C_1:=\sup_{n \geq 1} \mathcal E_{X_n}[f_n] = (\inf_{n \geq 1} R_n(1^n,3^n))^{-1}<\infty$. Pick a positive integer $n_1$ such that $\sum_{n=n_1+1}^\infty (6\lambda)^n < \frac 1{36C_1}$. Then for $n \geq n_1$,
\begin{equation} \label{eq5.5}
|f_n(1^n)-f_n(1^{n_1})|^2 \leq \mathcal E_{X_n}[f_n]R_{X_n}(1^n, 1^{n_1}) \leq C_1\sum_{k=n_1+1}^n (6\lambda)^k \leq \frac 1{36},
\end{equation}
which implies $f_n(1^{n_1}) \geq \frac 56$. Analogously we have $f_n(3^{n_1}) \leq \frac 16$. Let $m=n-n_1$. With a similar argument as in \eqref{eq5.5}, for $\zzz \in \{1,4\}^m$,
$$
|f_n(1^{n_1}\zzz)-f_n(1^{n_1})|^2 \leq \mathcal E_{X_n}[f_n]R_{X_n}(1^{n_1}\zzz, 1^{n_1}) \leq \frac 1{36},
$$
which implies $f_n(1^{n_1}\zzz) \geq \frac 23$. Analogously we have $f_n(3^{n_1}\www) \leq \frac 13$ for all $\www \in \{3,6\}^m$.
Now, for $\zzz = i_1 i_2 \cdots i_m \in \{1,4\}^m$, denote the word $j_1 j_2 \cdots j_m \in \{3,6\}^m$ with $j_k=i_k+2$ for all $k$ by $\zzz'$. Note that for each $\zzz \in \{1,4\}^m$, there is a horizontal path with length $3^n-1$ from $1^{n_1}\zzz$ to $3^{n_1}(\zzz')$. The resistance on such path is given by $R_{\mathcal J_n}(1^{n_1}\zzz, 3^{n_1}(\zzz')) = (3^n-1)(6\lambda)^n$. Counting the energy on these $2^m$ disjoint horizontal paths, we get
$$
C_1 \geq \mathcal E_{X_n}[f_n] \geq \sum_{\zzz \in \{1,4\}^m} \frac{[f_n(1^{n_1}\zzz)-f_n(3^{n_1}(\zzz'))]^2}{R_{\mathcal J_n}(1^{n_1}\zzz, 3^{n_1}(\zzz'))} \geq \frac{2^{n-n_1}}{9(3^n-1)(6\lambda)^n}
$$
for arbitrary $n \geq n_1$. Hence $\lambda \geq \frac 19$ and the claim follows. By Proposition \ref{th4.3}, we have $\beta_1^*=2$. \hfill  $\Box$}
\end{example}

\bigskip

\noindent {\bf Remark.} To investigate the situation that $\beta_1^*<\beta_3^*$, it is natural to study the products of self-similar sets. But in general, if $K_1$ and $K_2$ are connected self-similar sets, then the critical exponent of the product $K_1 \times K_2$ satisfies
$$\beta_1^* \leq \max\{\dim_H K_1,\dim_H K_2\}+1 \leq \dim_H K_1+ \dim_H K_2 = \alpha.$$
Although the criteria in the last section cannot be applied directly, it still has a similar link between the effective resistance of $\en_X$ and the energy on the product (see \cite{K} for more details). For example, in the product $[0,1] \times SG$, the limiting resistances $R^{(\lambda)}(i^\infty,j^\infty)$ have two critical exponents $\lambda_1^* = \frac 14$ and $\lambda_3^* = \frac 15$ for various $i,j$, while $2=\beta_1^*<\frac {\log 5}{\log 2}=\beta_3^*< \alpha = \frac {\log 6}{\log 2}$. With a similar technique as in Example \ref{ex5.3}, it follows that $\beta_1^*=2$ if one of $K_i$ is a unit interval. To generalize the results above, we may leave a conjecture as
$$
\beta_1^*(K_1 \times K_2) = \min\{\beta_1^*(K_1),\beta_1^*(K_2)\}, \quad \hbox{and} \quad \beta_3^*(K_1 \times K_2) = \max\{\beta_3^*(K_1),\beta_3^*(K_2)\}.
$$

\bigskip

\section{Remarks and open problems}
\label{sec:7}

\noindent The calculation of the critical exponents in Section \ref{sec:6} depends very much on the p.c.f.~property. It is challenging to find an effective technique to estimate the non-p.c.f.~sets like the Sierpi{\'n}ski carpet.

\medskip

In our discussions, we assumed the return ratio $\lambda \in (0,r^\alpha)$ (hence $\alpha < \beta_1^*$) in order to guarantee functions in the domain of the induced bilinear form on $K$ are continuous (Proposition \ref{th2.4}). While the condition is satisfied by the well-known fractals, it also excludes the situation that $\beta_1^* \leq \alpha$, which contains important examples (e.g., the classical domain, and product of fractals). We conjecture that the consideration in the paper is possible to adjust to this case. We also like to know if there is a nice sufficient condition for $\alpha <\beta_1^*$ based on the geometry of the self-similar sets.

\medskip

We call a self-similar set  $K$  {\it mono-critical} if it has a single critical exponent $\beta^* = \beta^*(K)$, i.e., $\beta^*=\beta_1^*=\beta_2^*=\beta_3^*$. It is known that all nested fractals, Cantor-type sets, and some non-p.c.f.~sets including Sierpi{\'n}ski carpet are mono-critical ( see \cite{BB1,BB2,BB3}). For these sets, the critical exponent plays an important role. It is well-known that $\Lambda_{2,2}^{\alpha,\beta^*/2}$ is trivial (see \cite{Jo,P1}) while $\Lambda_{2,\infty}^{\alpha,\beta^*/2}$ admits a local regular Dirichlet form on $L^2(K)$. On the other hand, it is constructed in \cite{GuL} a modified Vicsek set that is mono-critical; on this set, $\Lambda_{2,\infty}^{\alpha,\beta^*/2}$ is dense in $L^2(K, \nu)$, but is not dense in $C(K)$, and there is a local regular Dirichlet form on $K$ which does not define on $\Lambda_{2,\infty}^{\alpha,\beta^*/2}$ or satisfies the energy self-similar identity in \cite{Ki1}.

\medskip

In conclusion, the question of constructing  a local Dirichlet form  on a self-similar set  is still unsettled. It has much to do  with the functional behavior of the Besov spaces at the critical exponents. Our study offers an alternative approach of using the return ratio $\lambda$ of the random walk and the induced Dirichlet form to study these critical cases. It will be interesting to carry out this initiation to a greater extension.

\bigskip

\noindent {\bf Acknowledgements}: The authors would like to thank Professors A.~Grigoryan, J.X.~Hu and Dr.~Q.S.~Gu for many valuable discussions. They also thank Professor S.M.~Ngai for going through the manuscript carefully. Part of the work was carried out while the second author was visiting the University of Pittsburgh, he is grateful to Professors C.~Lennard and J.~Manfredi for the arrangement of the visit.

\bigskip
\bigskip

\noindent Shi-Lei Kong, Fakult{\"a}t f{\"u}r Mathematik, Universit{\"a}t Bielefeld, Postfach 100131, 33501 Bielefeld, Germany. \\
skong@math.uni-bielefeld.de

\bigskip
\noindent Ka-Sing Lau, Department of Mathematics, The Chinese University of Hong Kong, Hong Kong.\\
\& \ School of Mathematics and Statistics, Central China Normal University, Wuhan, 430079,
\ China.\\
\& Department of Mathematics, University of Pittsburgh, Pittsburgh, PA 15260, USA.\\
kslau@math.cuhk.edu.hk

\end{document}